\documentclass[reqno, 10pt]{amsart}

 \usepackage[top=2.8cm, bottom=2.7cm, left=3cm, right=3cm]{geometry}

\usepackage[centertags]{amsmath}
\usepackage{amsmath}
\usepackage[dutch,british]{babel}
\usepackage{amsfonts}
\usepackage{amssymb}
\usepackage{amsthm}

\usepackage{newlfont}
\usepackage{color}

\usepackage{authblk}

\usepackage{bbm}

\usepackage{stmaryrd}
\usepackage{mathrsfs}

\usepackage{fancyhdr}
\usepackage{graphicx}
\usepackage{fancybox}
\usepackage{setspace}

\usepackage{enumitem}
\usepackage{cleveref}

\theoremstyle{definition}

\theoremstyle{remark}

\theoremstyle{plain} 
\newtheorem{theorem}{Theorem}[section]
\newtheorem*{theorem*}{Theorem}
\newtheorem{lemma}[theorem]{Lemma}
\newtheorem{corollary}[theorem]{Corollary}
\newtheorem*{corollary*}{Corollary}
\newtheorem{proposition}[theorem]{Proposition}
\newtheorem*{proposition*}{Proposition}
\newtheorem{definition}[theorem]{Definition}
\newtheorem*{definition*}{Definition}
\newtheorem{assumption}[theorem]{Assumption}

\theoremstyle{definition} 

\newtheorem*{example*}{Example}
\newtheorem{remark}[theorem]{Remark}

\newtheorem*{remark*}{Remark}
\newtheorem*{remarks*}{Remarks}

\newcommand{\deq}{\mathrel{\mathop:}=}
\newcommand{\e}[1]{\mathrm{e}^{#1}}
\newcommand{\R} {\mathbb{R}}
\newcommand{\C} {\mathbb{C}}
\newcommand{\N} {\mathbb{N}}
\newcommand{\Z} {\mathbb{Z}}

\newcommand{\E} {\mathbb{E}}
\newcommand{\T} {\mathbb{T}}
\newcommand{\p} {\mathbb{P}}
\newcommand{\adj}{^*} 

\DeclareMathOperator{\diag}{diag}

\DeclareMathOperator{\Tr}{Tr}

\DeclareMathOperator{\im}{\mathrm{Im}}

\newcommand{\caD}{{\mathcal D}}

\newcommand{\caN}{{\mathcal N}}
\newcommand{\caO}{{\mathcal O}}

\newcommand{\caQ}{{\mathcal Q}}

\newcommand{\opunit}{\text{1}\kern-0.22em\text{l}}

\newcommand{\frc}{{\mathfrak c}}

\newcommand{\frf}{{\mathfrak f}}

\newcommand{\frn}{{\mathfrak n}}

\newcommand{\bsv}{{\boldsymbol v}}

\newcommand{\bsy}{{\boldsymbol y}}

\newcommand{\wt}{\widetilde}

\newcommand{\beq}{ \begin{equation} }
\newcommand{\eeq}{ \end{equation} }

\newcommand{\baq}{ \begin{eqnarray} }
\newcommand{\eaq}{ \end{eqnarray} }
\newcommand{\bet}{ \begin{theorem} }
\newcommand{\eet}{ \end{theorem} }

\newcommand{\lone}{\mathbbm{1}}

\newcommand{\ben}{\begin{arabicenumerate}}
\newcommand{\een}{\end{arabicenumerate}}

\newcommand{\dd}{\mathrm{d}}
\newcommand{\ii}{\mathrm{i}}

\renewcommand{\a}{\mathrm{a}}

\renewcommand{\P}{\mathbb{P}}

\newcommand{\fc}{\textsl{fc}}
\newcommand{\MP}{\textsl{MP}}
\newcommand{\GOE}{\textsc{goe}}

\numberwithin{equation}{section} 
\numberwithin{thm}{section}

\begin{document}
 \begin{minipage}{0.85\textwidth}
 \vspace{2.2cm}
 \end{minipage}
\begin{center}
\Large\bf
Tracy-Widom Distribution for the Largest Eigenvalue of Real Sample Covariance Matrices with General Population
\end{center}

\renewcommand{\thefootnote}{\fnsymbol{footnote}}	
\vspace{1.4cm}
\begin{center}
 \begin{minipage}{0.45\textwidth}
\begin{center}
Ji Oon Lee\footnotemark[1]  \\
\footnotesize { KAIST }\\
{\it jioon.lee@kaist.edu}
\end{center}
\end{minipage}
\begin{minipage}{0.45\textwidth}
 \begin{center}
Kevin Schnelli\footnotemark[2]\\
\footnotesize 
{IST Austria}\\
{\it kevin.schnelli@ist.ac.at}
\end{center}
\end{minipage}
\footnotetext[1]{Partially supported by TJ Park Junior Faculty Fellowship.}
\footnotetext[2]{Supported by ERC Advanced Grant RANMAT No.\ 338804. }

\end{center}

\vspace{1.1cm}

\begin{center}
 \begin{minipage}{0.9\textwidth}
\small
\hspace{10pt} We consider sample covariance matrices of the form $\caQ=(\Sigma^{1/2}X)(\Sigma^{1/2} X)\adj$, where the sample~$X$ is an $M\times N$ random matrix whose entries are real independent random variables with variance~$1/N$ and where~$\Sigma$ is an $M\times M$ positive-definite deterministic matrix. We analyze the asymptotic fluctuations of the largest rescaled eigenvalue of $\caQ$ when both $M$ and $N$ tend to infinity with $N/M\to d\in(0,\infty)$. For a large class of populations $\Sigma$ in the sub-critical regime, we show that the distribution of the largest rescaled eigenvalue of $\caQ$ is given by the type-1 Tracy-Widom distribution under the additional assumptions that (1) either the entries of $X$ are i.i.d.\ Gaussians or (2) that~$\Sigma$ is diagonal and that the entries of $X$ have a subexponential decay. 

 \end{minipage}
\end{center}
 
 \vspace{7mm}
 
\thispagestyle{headings}

\section{Introduction}

Covariance matrices are fundamental objects in multivariate statistics whose study is an integral part of various fields such as signal processing, genomics, financial mathematics, {\it etc.}. Sample covariance matrices are the simplest estimators for population covariance matrices: The population covariance matrix of a mean-zero random variable $\bsy\in\R^M$ is $\Sigma\deq\E \bsy\bsy'$. Given $N$ independent samples $(\bsy_1,\ldots,\bsy_N)$ of~$\bsy$,~$\Sigma$ may be estimated through the sample covariance matrix $\caQ\deq\frac{1}{N}\sum_{i=1}^N \bsy_i{\bsy_i}'$. Indeed, since $\E \caQ=\Sigma$, $\caQ$ converges, for fixed $M$, almost surely to $\Sigma$ as $N$ tends to infinity. However, in many modern applications the population size~$M$ may be as large or even large than $N$ and, hence, one may take~$M$ and~$N$ simultaneously to infinity in an asymptotic analysis. In this setting,~$\Sigma$ cannot be estimated trough~$\caQ$ due to the high dimensionality. Yet, some properties of $\Sigma$ may be inferred 
from spectral statistics of~$\caQ$, {\it e.g.,} the limiting behavior of the largest eigenvalues of $\caQ$ is frequently used in hypothesis testing for the structure of~$\Sigma$.

In this paper we investigate the limiting behavior of the largest eigenvalues of the form
\begin{align}\label{le Q intro}
 \caQ= (\Sigma^{1/2}X)(\Sigma^{1/2}X)\adj\,,
\end{align}
where the sample or data matrix, $X$, is an $M \times N$ matrix whose entries are a collection of independent real or complex random variables of variance~$1/N$ and where the general population covariance, $\Sigma$, is an $M \times M$ real positive-definite deterministic matrix. We are interested in the high-dimensional case, where $\widehat d\deq N/M \to d \in (0, \infty)$, as $N \to \infty$. We further mainly focus on the real setting, where $X$ is a real data matrix, since, mathematically, the complex case is easier to deal with. Also, the real case is of primary interest in statistics, although complex data matrices arise in some applications. For detailed discussions of this model, we refer to, {\it e.g.},~\cite{BaiS,Jon2,Jo,Ek,BPZ2}. In Subsection~\ref{application discussion}, we outline an application of this model. We denote the eigenvalues of~$\caQ$ and~$\Sigma$ in decreasing order by~$(\mu_i)_{i=1}^N$ and~$(\sigma_m)_{m=1}^M$ respectively.

The main results of this paper show that the limiting distribution of the largest rescaled eigenvalue of~$\caQ$ is given by the Tracy-Widom distribution, {\it i.e.},
\begin{align}\label{le what we show}
 \lim_{N\to\infty}\mathbb{P}\left(\gamma_0 N^{2/3}{(\mu_1-E_+)}\le s \right)=F_1(s)\,,\qquad\qquad (s\in\R)\,,
\end{align}
where $\gamma_0\equiv\gamma_0(N)$ and $E_+\equiv L_+(N)$ depend only on the sequence $(\sigma_m)_{m=1}^M$ and the ratio~$\widehat d$. Here,~$F_1$ denotes the cumulative distribution function (CDF) of the type-$1$ Tracy-Widom distribution~\cite{TW1,TW2} which arises as the limiting CDF of the largest rescaled eigenvalue of the Gaussian orthogonal ensemble (GOE). More precisely, we show that~\eqref{le what we show} holds in the ``sub-critical regime'' where the largest eigenvalues of~$\Sigma$ are close to the bulk of the spectrum of~$\Sigma$ (for a precise statement see Assumption~\ref{assumptions for local MP law} below) and if either of the followings holds:
\begin{itemize}[noitemsep,topsep=0pt,partopsep=0pt,parsep=0pt]
 \item[$(1)$] the entries of $X$ are i.i.d.\ Gaussians (Corollary~\ref{cor main}), or
 \item[$(2)$] the general population~$\Sigma$ is diagonal and the entries of $X$ have a subexponential decay (Theorem~\ref{thm main}).
\end{itemize}

Our results are also valid in the setting of complex data matrices $X$. In that setup, one replaces~$F_1$ in~\eqref{le what we show} by~$F_2$, the CDF of the type-$2$ Tracy-Widom distribution which arises as the limiting CDF of the largest rescaled eigenvalue of the Gaussian unitary ensemble (GUE).

To situate our result in the literature, we first recall that the limiting spectral distribution of the model~\eqref{le Q intro} was derived for general $\Sigma$ by Marchenko and Pastur~\cite{MP}. When~$X$ has i.i.d.\ Gaussian entries,~$\caQ$ is called a Wishart matrix. For Wishart matrices with identity population covariance, often referred to the null case, it is well known that the limiting distribution of the largest rescaled eigenvalue coincides with the corresponding distribution of the GOE and GUE respectively: in the null case,~\eqref{le what we show} was obtained in~\cite{Jo} for real Wishart matrices and in~\cite{J} for complex Wishart matrices.

In the non-null case, where~$\Sigma$ is not a multiple of the identity matrix, first results were obtained for spiked population models introduced in~\cite{Jo}, where~$\Sigma$ is a finite rank perturbation of the identity matrix. Complex spiked Wishart matrices were studied in~\cite{BBP}, where an interesting phase transition in the asymptotic behavior of the largest rescaled eigenvalue as a function of the spikes was observed. In particular, it was shown that the largest rescaled eigenvalue follows the Tracy-Widom distribution~$F_2$ in the sub-critical regime, {\it i.e.},  for small finite rank perturbations. These results rely on an explicit formula--the Baik-Ben Arous-Johansson-P\'{e}ch\'{e} (BBJP)-formula--for the joint eigenvalue distribution of complex Wishart matrices. For real Wishart matrices, the counterpart of the BBJP-formula is not available, due to the lack of an analogue of the Harish-Chandra-Itzykson-Zuber 
integral for the orthogonal group. Relying on quite different methods, almost sure convergence of the largest eigenvalues was derived in~\cite{BS} and Tracy-Widom fluctuations of the largest eigenvalue of spiked population models were obtained in~\cite{FP}. The equivalent results of the aforementioned phase transition for finite rank perturbations were obtained in the real setting in~\cite{BV1,BV2,Mo,F1}.

In the general non-null case sufficient conditions for the validity of~\eqref{le what we show} in the sub-critical regime were given in~\cite{Ek} for the non-singular case $d\in(0,\infty)$, $d\not=1$, and in~\cite{O} for the singular case $d=1$. Yet, these results rely on the BBJP-formula and are thus limited to complex Wishart matrices. Corollary~\ref{cor main} below establishes under similar assumptions the limiting behavior~\eqref{le what we show} for real Wishart matrices with general population.                                                                                                                                                                                                                                                                                                                        

The aforementioned results are believed to be universal in the sense that they are independent of the details of the distributions of the entries of $X$ (provided they decay sufficiently fast). This phenomenon is referred to as edge universality. It was established in the null case in~\cite{So,Pe1,FS} for symmetric distributions and subsequently in~\cite{Wa} for distributions with vanishing third moment. This third moment condition was removed in~\cite{PY}. For spiked sample covariance matrices, universality results were obtained in~\cite{FP} under the assumption that  the entries' distribution of $X$ are symmetric. This condition was removed in~\cite{BKYY}. For full rank deformed populations matrices $\Sigma$, universality results were obtained in~\cite{BPZ2} under the assumption that~$\Sigma$ is either diagonal or that the first four moments of the entries' distribution of $X$ match those of the standard Gaussian distribution in case~$\Sigma$ is non-diagonal. Recently, the edge universality was 
established in~\cite{KY} for general $\Sigma$. Once the edge 
universality for general sample covariance matrices has been established, the limiting CDF of the rescaled largest eigenvalue may then be identified in the complex setting with~$F_2$ via the results of~\cite{Ek,O}. In the real setting, this identification was only possible in the null case and finite rank deformations thereof. Our main new results allow this identification in the real setting with general population covariances, {\it i.e.,} it allows to identify $F_1$ as the limiting CDF of the rescaled largest eigenvalue.

Our proof of~\eqref{le what we show} is based on a comparison of Green functions. Discrete Green function comparison via Lindeberg's replacement strategy~\cite{C,TV2} was used to prove the edge universalities of Wigner matrices~\cite{EYY,TV2} and of null sample covariance matrices~\cite{PY}. Continuous Green function comparison was used to establish CLT results for linear statistics of null sample covariance matrices~\cite{LP}, and, more recently, to derive estimates on the Green function itself, {\it i.e.,} local laws, for non-null sample covariance matrices~\cite{KY}. However, as for the deformed Wigner matrices considered in~\cite{LSY}, a direct application of discrete or continuous Green function comparison does not work for non-null sample covariance matrices. We thus adopt the new approach developed in~\cite{LSY}: we consider a continuous interpolation between the given sample covariance matrix and a null sample covariance matrix. We follow the associated Green function flow and estimate its change 
over time. This change is then offset by renormalizing the matrix.

Our analysis requires as an {\it a priori} ingredient a local law for the Green function, {\it i.e.}, an optimal estimate on the entries of the Green function on scales slightly below $N^{-2/3}$ at the upper edge (see Lemma~\ref{lemma local law} below for a precise statement). Optimal local laws in the bulk and at the edges of the spectrum were obtained for Wigner matrices in~\cite{EYY1,EYYBernoulli,EKYY4}. Using a similar approach, optimal local laws for sample covariance matrices with $\Sigma=\lone$ were obtained in~\cite{PY}; see also~\cite{BKYY1,ESYY}. These results were extended to sample covariance matrices with general population under a four moment matching condition in~\cite{BPZ2}. The four moment matching condition was very recently removed in~\cite{KY}.

This paper is organized as follows: In Section 2, we define the model, present the main results of the paper and outline some applications. In Sections 3, we collect the tools and known results we need in our proofs. In Section 4, we prove the main theorems using our essential new technical result, Proposition~\ref{prop green}, the Green function comparison theorem at the edge. In Sections 5 and 6, we outline the ideas of the proof of the Green function comparison theorem. Its technical details can be found in the Appendices~\ref{sec:expansion of GiaGai},~\ref{sec:optical theorems} and~\ref{sec:proof of getting rid of the rest}. Some results required in these appendices are adaptations from~\cite{LSY}.

\section{Definitions and Main Result}

\subsection{Sample covariance matrix with general population}

\begin{definition} \label{assumption sample}
Let $X=(x_{ij})$ be an $M\times N$ matrix whose entries $\{x_{ij}\,:\, 1\le i\le M,\, 1\le j\le N\}$ are a collection of independent real random variables such that 
\begin{align}
 \E\, x_{ij}=0\,,\qquad \quad \E\,|x_{ij}|^2=\frac{1}{N}\,.
\end{align}
Moreover, we assume that $(\sqrt{N}x_{ij})$ have a subexponential tail, {\it i.e.}, there are~$C$ and~$\vartheta>0$ such that
\begin{align}\label{le subexponential decay}
 \mathbb{P}(|\sqrt{N}x_{ij}|> t)\le C\e{-t^{\vartheta}}\,,
\end{align}
for all~$i,j$. 

Further, $M\equiv M(N)$ with
\begin{align}
\widehat d=\frac{N}{M}\to d\in (0,\infty)\,,
\end{align}
as $N\to \infty$. For simplicity, we assume that~$N/M$ is constant, hence we use~$d$ instead of~$\widehat d$.
\end{definition}

Note that we do not require in Definition~\ref{assumption sample} that the entries or columns of $X$  are identically distributed.

Let~$\Sigma$ be an~$M\times M$ real positive-definite deterministic matrix. We denote by~$\widehat \rho$ the empirical eigenvalue distribution of~$\Sigma$,  {\it i.e.}, if we let~$\sigma_1 \geq \sigma_2 \geq \cdots \geq \sigma_M\ge 0$ be the eigenvalues of~$\Sigma$, then
\begin{align}\label{le spectral distriution of sigma}
 \widehat\rho\deq\frac{1}{M}\sum_{j=1}^M\delta_{\sigma_j}\,.
\end{align}
We then form the sample covariance matrix
\begin{align}
 \caQ\deq (\Sigma^{1/2}X)(\Sigma^{1/2}X)\adj \,,
\end{align}
and denote its eigenvalues in decreasing order by $\mu_1\ge \mu_2\ge\ldots\ge\mu_M$. Note that the $M\times M$ matrix~$\caQ$ and the $N\times N$ matrix
\begin{align}
 Q\deq X\adj \Sigma X
\end{align}
share the same non-zero eigenvalues. Since we are interested in behavior of the largest eigenvalues, we focus on $Q$ since it is for technical reasons more amenable than $\caQ$. With some abuse of terminology, we also call $Q$ a sample covariance matrix and we denote its $M$ largest eigenvalues by $(\mu_i)_{i=1}^M$, too.

\subsection{Deformed Marchenko Pastur law}\label{subsection deformed MP law}
Assuming that the empirical spectral distribution $\widehat\rho$ of $\Sigma$ converges weakly to some distribution~$\rho$, it was shown in~\cite{MP} that the empirical eigenvalue distribution of~$Q$ converges weakly in probability to a deterministic distribution, $\rho_{\fc}$, referred to as the ``deformed Marchenko-Pastur law'' below, which depends on~$\rho$ and the ratio~$d$. It can be described in terms of its Stieltjes transform: For a (probability) measure~$\omega$ on the real line we define its Stieltjes transform by
\begin{align}
 m_{\omega}(z)\deq\int\frac{\dd\omega(v)}{v-z}\,,\qquad\qquad (z=E+\ii\eta\in\C^+)\,.
\end{align}
Here and below, we write $z=E+\ii\eta$, with $E\in\R$, $\eta\ge 0$. Note that $m_{\omega}$ is an analytic function in the upper half-plane and that $\im m_\omega(z)\ge 0$, $\im z> 0$. Assuming that $\omega$ is absolutely continuous with respect to Lebesgue measure, we can recover the density of $\omega$ from $m_{\omega}$ by the inversion formula
\begin{align}\label{stieltjes inversion formula}
 \omega(E)=\lim_{\eta\searrow 0}\frac{1}{\pi}\im m_{\omega}(E+\ii \eta)\,,\qquad\qquad (E\in\R)\,.
\end{align}
We use the same symbols to denote measures and their densities.

Choosing $\omega$ to be the standard Marchenko-Pastur law $\rho_{\MP}$, the Stieltjes transform $m_{\rho_{\MP}}\equiv m_{\MP}$ can be computed explicitly and one checks that~$m_{\MP}$ satisfies the relation
\begin{align}
 m_{\MP}(z)=\frac{1}{-z+d^{-1}\frac{1}{ m_{\MP}(z)+1}}\,,\quad \im  m_{\MP}(z)\ge 0\,,\qquad (z\in\C^{+})\,.
\end{align}

The deformed Marchenko-Pastur law $\rho_{\fc}$ is defined as follows. Assume that~$\widehat\rho$ converges weakly to~$\rho$ as~$N$ goes to infinity. Then the Stieltjes transform of the deformed Marchenko-Pastur law, $m_{\fc}$, is obtained as the unique solution to the self-consistent equation
\begin{align}\label{defining equation for mfc}
 m_{\fc}(z)=\frac{1}{-z+d^{-1}\int\frac{t}{tm_{\fc}(z)+1}\dd\rho(t)}\,,\quad \im m_{\fc}(z)\ge 0\,,\qquad(z\in\C^+)\,.
\end{align}
It is well known~\cite{MP} that the functional equation~\eqref{defining equation for mfc} has a unique solution that is uniformly bounded on the upper half-plane. The density of the deformed Marchenko-Pastur law~$\rho_{\fc}$ is obtained from~$m_{\fc}$ by the Stieltjes inversion formula~\eqref{stieltjes inversion formula}. The measure~$\rho_{\fc}$ has been studied in~\cite{SC}, {\it e.g.}, it was shown that the density of~$\rho_{\fc}$ is an analytic function inside its support. The measure~$\rho_{\fc}$ is also called the multiplicative free convolution of the Marchenko-Pastur law and the measure~$\rho$; we refer to, {\it e.g.},~\cite{VDN,AGZ}.

For finite $N$, we let $\widehat m_{\fc}$ denote the unique solution to
\begin{align} \label{hat mfc}
 \widehat m_{\fc}(z)=\frac{1}{-z+d^{-1}\int\frac{t}{t\widehat m_{\fc}(z)+1}\dd\widehat\rho(t)}\,,\quad \im \widehat m_{\fc}(z)\ge 0\,,\qquad (z\in\C^{+})\,,
\end{align}
and let $\widehat\rho_{\fc}$ denote the measure obtained from $\widehat m_{\fc}(z)$ through~\eqref{stieltjes inversion formula}. It is easy to check that~$\widehat\rho_{\fc}$ is a well-defined probability measure with a continuous density.

The rightmost endpoint of the support of $\widehat\rho_{\fc}$ is determined as follows. Define~$\xi_+$ as the largest solution to
\begin{align} \label{xi+}
 \int\left(\frac{t \xi_+}{1-t \xi_+} \right)^2\dd\widehat\rho(t)=d\,,
\end{align}
with $d=\frac{N}{M}$. Note that $\xi_+$ is unique and that $\xi_+\in[0, \sigma_1^{-1}]$. We also introduce~$E_+$ by setting
\begin{align} \label{E+}
E_+\deq\frac{1}{\xi_+}\left(1+d^{-1}\int\frac{t \xi_+}{1-t \xi_+}\dd\widehat\rho(t)\right)\,.
\end{align}
Considering the imaginary part of~\eqref{hat mfc} in the limit~$\eta\searrow 0$, one infers~\cite{SC} that the rightmost edge of~$\widehat\rho_{\fc}$, {\it i.e.}, the rightmost endpoint of the support of $\widehat\rho_{\fc}$, is given by $E_+$ and that
\begin{align}
\xi_+=-\lim_{\eta\to 0}\widehat m_{\fc}(E_+ +\ii\eta)=-\widehat m_{\fc}(E_+)\,.
\end{align}

The following assumption is required to establish our main results. It appeared previously in~\cite{Ek,BPZ2}.

\begin{assumption}\label{assumptions for local MP law}
Let $\sigma_1\ge\sigma_2\ge\ldots \ge\sigma_M$ denote the eigenvalues of $\Sigma$. Then, we assume that
$\liminf_N \sigma_M>0$, $\limsup_N \sigma_1<\infty$ and
\begin{align}\label{Sigma assumption}
\limsup_N \sigma_1 \xi_+<1\,.
\end{align}

\end{assumption}

\begin{remark}
 We remark that Assumption~\ref{assumptions for local MP law} was used in \cite{BPZ2,KY} to derive the local deformed Marchenko-Pastur law for $Q$. The inequality~\eqref{Sigma assumption} guarantees that the distribution~$\widehat \rho_{\fc}(E)$ exhibits a square-root type behavior at the rightmost endpoint of its support; see Lemma~\ref{lem:square root for hat mfc} below.
\end{remark}

\subsection{Main result}

The main result of this paper is as follows:
\begin{theorem} \label{thm main}
Let $Q = X^* \Sigma X$ be an $N \times N$ sample covariance matrix with sample~$X$ and population~$\Sigma$, where $X$ is a real random matrix satisfying the assumptions in Definition~\ref{assumption sample} and $\Sigma$ is a real diagonal deterministic matrix satisfying Assumption~\ref{assumptions for local MP law}. Recall that $F_1$ denotes the cumulative distribution function of the type-1 Tracy-Widom distribution.

Let~$\mu_1$ be the largest eigenvalue of~$Q$. Then, there exist $\gamma_0\equiv \gamma_0(N)$ depending only on the empirical eigenvalue distribution $\widehat\rho$ of $\Sigma$ and the ratio~$d$ such that the distribution of the largest rescaled eigenvalue of~$Q$ converges to the Tracy-Widom distribution, {\it i.e.},
\beq \label{eq:main}
\lim_{N \to \infty} \p \left( \gamma_0 N^{2/3} \big( \mu_1 - E_+ \big) \leq s \right) = F_1 (s)\,,
\eeq
for all $s\in\R$, where $E_+\equiv E_+(N)$ is given in~\eqref{E+}.
\end{theorem}

\begin{remark}
The scaling factor $\gamma_0\equiv \gamma_0(N)$ is given by~\cite{Ek}
\begin{align}\label{le gamma0}
\frac{1}{\gamma_0^{3}} = \frac{1}{d} \int \left( \frac{t}{1 - t \xi_+} \right)^3 \dd \widehat{\rho}(t) + \frac{1}{\xi_+^3}\,.
\end{align}
It follows from Assumption~\ref{assumptions for local MP law} that $\gamma_0=O(N^0)$.
\end{remark}

\begin{remark} \label{k main}\label{remark k largest eigenvalues}
Theorem \ref{thm main} can be extended to correlation functions of the extremal eigenvalues as follows: Let $W^{\GOE}$ be an $N\times N$ random matrix belonging to the Gaussian Orthogonal Ensemble (GOE); see~\cite{AGZ,Mbook}. The joint distributions of $\mu_1^{\GOE} \geq \mu_2^{\GOE} \geq \cdots \geq \mu_N^{\GOE}$, the eigenvalues of $W^{\GOE}$, are explicit and the joint distribution of the $k$ largest eigenvalues can be written in terms of the Airy kernel~\cite{F} for any fixed~$k$.  The generalization of~\eqref{eq:main} to the $k$ largest eigenvalues of~$Q$ then reads
\begin{align} \label{eq: k main}
\lim_{N \to \infty} \p \left(\left(\gamma_0 N^{2/3} \big( \mu_i - E_+ \big) \leq s_i\right)_{1\le i\le k} \right)= \lim_{N \to \infty} \p \left(\left( N^{2/3} \big( \mu_i^{\GOE} - 2 \big) \leq s_i \right)_{1\le i\le k}\right)\,,
 \end{align}
for all $s_1,s_2,\ldots, s_k\in\R$.
\end{remark}

If the entries of $X$ are Gaussian, the result in Theorem~\ref{thm main} holds for general, non-diagonal $\Sigma$.

\begin{corollary} \label{cor main}
Let $Q = X^* \Sigma X$ be an $N \times N$ sample covariance matrix with sample $X$ and general population $\Sigma$, where $X$ is a real random matrix with independent Gaussian entries satisfying the assumptions in Definition~\ref{assumption sample} and $\Sigma$ is a real positive-definite deterministic matrix satisfying Assumption~\ref{assumptions for local MP law}. Let~$\mu_1$ be the largest eigenvalue of $Q$. 

Then the distribution of the largest rescaled eigenvalue of $Q$ converges to the type-$1$ Tracy-Widom distribution, {\it i.e.},
\beq\label{equation cor main}
\lim_{N \to \infty} \p \left(\gamma_0 N^{2/3} \big( \mu_1 - E_+ \big) \leq s \right) = F_1 (s)\,,
\eeq
for all $s\in\R$, where $E_+\equiv E_+(N)$ is given in~\eqref{E+} and $\gamma_0=\gamma_0(N)$ is given in~\eqref{le gamma0}.
\end{corollary}

\begin{remark}
For non-Gaussian $X$ and general off-diagonal $\Sigma$, we can combine our results with edge universality results in~\cite{BPZ2,KY} to identify the Tracy-Widom distribution for the largest eigenvalues. 
\end{remark}

\subsection{Applications}\label{application discussion}
In this subsection, we briefly discuss possible applications of our results to statistics. For a general overview of applications of random matrix theory to statistical inference we refer to the review~\cite{Jon2}, where the following application to signal detection problems is described. 

Consider a signal-plus-noise vector 
\begin{align}
\boldsymbol{y}\deq D\boldsymbol{s}+\Sigma_{0}^{1/2} \boldsymbol{z}
\end{align}
of dimension $M$, where $\boldsymbol{s}$ is a $k$-dimensional real mean-zero signal vector with population covariance matrix $S$, $D$ is a $M\times k$ real deterministic matrix which is of full column rank, $\boldsymbol{z}$ is an $M$-dimensional real or complex random vector and $\Sigma_0$ is an $M\times M$ deterministic positive-definite matrix. In many situations $\boldsymbol{z}$ is assumed to be Gaussian. Assuming that the signal vector, $D\boldsymbol{s}$, and the noise vector, $\Sigma_0^{1/2}\boldsymbol{z}$, are independent the population covariance matrix, $\Sigma$, of $\boldsymbol{y}$ is given by
\begin{align}
\Sigma=DSD^*+\Sigma_0\,. 
\end{align}
A fundamental question is to detect signals from given data, {\it e.g.}, from independent samples, $(\boldsymbol{y}_i)_{i=1}^N$, of $\boldsymbol{y}$ and its associated sample covariance matrix~$\caQ$. A first step in this analysis is to determine whether there is any signal present that can be detected at all. Once signals are detected, one is led to estimate~$k$. More precisely, the first aim is to test, with general correlation noise $\Sigma_0^{1/2}\boldsymbol{z}$, whether there is no signal present, {\it i.e.,} to test the null hypothesis $k=0$ against the alternative hypothesis $k>1$. For details and results in the classical setting of large sample size $N$ and low dimensionality~$M$ see~\cite{kay}. 

In the high dimensional setup, the described signal detection problem was considered in~\cite{EN,Bia}, where the usability and performance of the largest eigenvalue as a test statistics was discussed in the presence of white Gaussian noise, {\it i.e.,} $\Sigma_0=\lone$ and Gaussian $\boldsymbol{z}$. In~\cite{RRNJWS}, the detection problem in presence of correlated Gaussian noise, {\it i.e.,} $\Sigma_0\not=\lone$ and Gaussian $\boldsymbol{z}$, was discussed at length. We also refer to~\cite{Vi} for further developments.

In the discussion above, it was implicitly assumed that $\Sigma_0$ is known {\it a priori}. For many real systems, however, $\Sigma_0$ is usually unknown. In particular, $\zeta_+$, $E_+$ and $\gamma_0$ of, {\it e.g.,}~\eqref{eq:main}, are unknown and the largest eigenvalue $\mu_1$ may not directly be used as a test statistics. Following~\cite{O2}, we observe that~$E_+$ and~$\gamma_0$ are eliminated under the null hypothesis in the test statistics $R\deq (\mu_1-\mu_2)/(\mu_2-\mu_3)$, with $\mu_2$, $\mu_3$, the second, respectively third largest eigenvalue of~$\caQ$. On the other hand, the limiting distribution of $R$ is determined by the Tracy-Widom-Airy statistics under the null hypothesis as stated in Remark~\ref{remark k largest eigenvalues} above. In fact, in the complex setting the test statistics~$R$ was shown~\cite{O2} to be asymptotically pivotal under the null hypothesis and one expects the same results to hold in the real setting. While there is no explicit formula for the limiting distribution of 
$R$, it may be effectively approximated by using numerics for extremal eigenvalues of GOE, respectively GUE, matrices.

\section{Preliminaries}

\subsection{Notations} \label{notation}
We first introduce a notation for high-probability estimates which is suited for our purposes. A slightly different form was first used in~\cite{EKY}.
\begin{definition}\label{definition of stochastic domination}
Let
\begin{align*}
 X=(X^{(N)}(u)\,:\, N\in\N\,, u\in U^{(N)})\,,\quad Y=(Y^{(N)}(u)\,:\, N\in\N\,,\,u\in U^{(N)})
\end{align*}
be two families of nonnegative random variables where $U^{(N)}$ is a possibly $N$-dependent parameter set. We say that $Y$ stochastically dominates $X$, uniformly in~$u$, if for all (small) $\epsilon>0$ and (large) $D>0$,
\begin{align}
 \sup_{u\in U^{(N)}}\P\left[X^{(N)}(u)>N^{\epsilon} Y^{(N)}(u) \right]\le N^{-D}\,,
\end{align}
for sufficiently large $N\ge N_0(\epsilon,D)$. If $Y$ stochastically dominates $X$, uniformly in~$u$, we write $X \prec Y$. If for some complex family $X$ we have $|X|\prec Y$ we also write $X=\caO(Y)$. 
\end{definition}
The relation $\prec$ is a partial ordering: it is transitive and it satisfies the arithmetic rules of an order relation, {\it e.g.}, if $X_1\prec Y_1$ and $X_2\prec Y_2$ then $X_1+X_2\prec Y_1+Y_2$ and $X_1 X_2\prec Y_1 Y_2$. Further assume that $\Phi(u)\ge N^{-C}$ is deterministic and that~$Y(u)$ is a nonnegative random variable satisfying $\E [Y(u)]^2\le N^{C'}$ for all~$u$. Then $Y(u) \prec \Phi(u)$, uniformly in $u$, implies $\E [Y(u)] \prec \Phi(u)$, uniformly in~$u$.

We use the symbols $O(\,\cdot\,)$ and $o(\,\cdot\,)$ for the standard big-O and \mbox{little-o} notation. The notations $O$, $o$, $\ll$, $\gg$, refer to the limit $N\to \infty$ unless otherwise stated. Here $a\ll b$ means $a=o(b)$. We use~$c$ and~$C$ to denote positive constants that do not depend on $N$, usually with the convention $c\le C$. Their value may change from line to line. We write $a\sim b$, if there is $C\ge1$ such that $C^{-1}|b|\le |a|\le C |b|$.

Finally, we use double brackets to denote index sets, {\it i.e.}, 
$$
\llbracket n_1, n_2 \rrbracket \deq [n_1, n_2] \cap \Z\,,
$$
for $n_1, n_2 \in \R$.

\subsection{Local deformed Marchenko-Pastur law}

For small positive $\frc, \epsilon$ and sufficiently large $C_+$, ($C_+>E_+$), we define the domain,  $\caD(\frc, \epsilon)$, of the spectral parameter $z$ by 
\begin{align*}
\caD(\frc, \epsilon)\deq\left\{z=E+\ii\eta\in\C^+\,:\, E_+-\frc\le E\le C_+\,, N^{\epsilon} N^{-1}\le \eta\le 1  \right\}\,.
\end{align*}

Let $\kappa\equiv\kappa_E\deq|E- E_+|$. Then we have the following results:

\begin{lemma}[Theorem 3.1 in~\cite{BPZ2}] \label{lem:square root for hat mfc}
 Under Assumption~\ref{assumptions for local MP law}, there is $c>0$ such that
\begin{align}
 \widehat\rho_{\fc}(E)\sim\sqrt{E_+-E}\,,\qquad\qquad (E\in[E_+-2c,E_+])\,.
\end{align}
The Stieltjes transform $\widehat m_{\fc}(z)$ of $\widehat\rho_{\fc}$ satisfies the following.
\begin{itemize}
 \item[$i.$]
For $z\in\caD(c,0)$,
\begin{align}
 |\widehat m_{\fc}(z)|\sim 1\,.
\end{align}
\item[$ii.$]
For $z\in\caD(c,0)$,
\begin{align}
 \im \widehat m_{\fc}(z)\sim\begin{cases}
                            \frac{\eta}{\sqrt{\kappa+\eta}}\,,\qquad& \mathrm{if}\,\,E\ge E_++\eta\,,\\
\sqrt{\kappa+\eta}\,,&\mathrm{if}\,\, E\in[E_+-c, E_++\eta)\,.
                           \end{cases}
\end{align}

\end{itemize}

\end{lemma}

We introduce the $z$-dependent control parameter, $\Psi(z)$, by setting
\begin{align}\label{le definition of the psi}
 \Psi\equiv\Psi(z)\deq\left(\frac{\im \widehat m_{\fc}(z)}{N\eta}\right)^{1/2}+\frac{1}{N\eta}\,.
\end{align}
We remark that, for $z = E + \ii \eta$ with $\kappa_E \leq N^{-2/3 + \epsilon}$ and $\eta = N^{-2/3 - \epsilon}$, we have 
$$
\Psi \leq C N^{-2/3 + \epsilon}\,.
$$
Define the Green function $G_Q=((G_Q)_{ij})$ by
\begin{align}
G_Q (z) \deq (Q -z)^{-1}\,,\qquad\qquad (z\in\C^+)\,,
\end{align}
and denote its average by
\begin{align}
m_Q(z) \deq \frac{1}{N} \mathrm{Tr }\, G_Q(z)\,,\qquad\qquad(z\in\C^+)\,.
\end{align}

Recall that $\mu_1$ denotes the largest eigenvalue of the sample covariance matrix $Q$. We have the following local law from~\cite{BPZ2}.

\begin{lemma}[Theorem 3.2 and Theorem 3.3 in~\cite{BPZ2}]\label{lemma local law}
Under the Assumption~\ref{assumptions for local MP law}, we have, for any sufficiently small $\epsilon>0$,
\begin{align}
 |m_Q(z)-\widehat m_{\fc}(z)|\prec \frac{1}{N\eta}\,,\qquad\quad \max_{i,j}|(G_Q)_{ij}(z)-\delta_{ij}\widehat m_{\fc}(z)|\prec\Psi(z)\,,
\end{align}
 uniformly in $z$ on $\caD(\epsilon,c)$, where $c$ is the constant in Lemma~\ref{lem:square root for hat mfc}. Moreover, we have
\begin{align}
 |\mu_1- E_+|\prec N^{-2/3}\,,
\end{align}
where $E_+$ is given in~\eqref{E+}.
\end{lemma}

\subsection{Density of states}

 In this subsection, we explain how the distribution of the largest eigenvalues of~$Q$ can be related to~$m_Q(z)$ for appropriately chosen~$z$. The arguments given here are small modifications of the methods presented in~\cite{EYY,PY,LSY}.

 Recall the definition of the scaling factor $\gamma_0$ in \eqref{le gamma0}. We set
\begin{align}
 T\deq \gamma_0\Sigma\,,
\end{align}
and define
\begin{align}
\wt Q \deq X\adj TX\,.
 \end{align}
We denote by $m_{\wt Q}$ the averaged Green function of $\wt Q$, {\it i.e.,}
\begin{align}
m_{\wt Q} (z) \deq \frac{1}{N} \mathrm{Tr } (\wt Q - z)^{-1}\,,\qquad\qquad (z\in C^+)\,.
\end{align}
Let $\wt \mu_1 \geq \wt \mu_2 \geq \cdots \geq \wt \mu_N$ be the eigenvalues of $\wt Q$. Let $L_+ \deq \gamma_0 E_+$ and observe that from Lemma~\ref{lemma local law} we have
\beq
|\wt \mu_1 - L_+ |\prec N^{-2/3}\,.\nonumber
\eeq	
Thus, we may assume in~\eqref{eq:main} that $|s| \prec 1$.

Fix $E_*$ such that
\beq
E_* - L_+ \prec N^{-2/3}, \qquad\quad \lone(\mu_1 - E_*>0) \prec 0\,. \nonumber
\eeq
We note that the choice of $E_*$ guarantees that the probability of the event $\{\mu_1 > E_*\}$ is negligible. For $E$ satisfying
\beq \label{def_E}
|E- L_+| \prec N^{-2/3}\,,
\eeq
we let
\beq
\chi_E \deq \lone_{[E, E_*]}\,.\nonumber
\eeq
We also define the Poisson kernel, $\theta_{\eta}$, for $\eta > 0$,
\beq
\theta_{\eta} (x) \deq \frac{\eta}{\pi (x^2 + \eta^2)} = \frac{1}{\pi} \im \frac{1}{x - \ii \eta}\,.\nonumber
\eeq
Introduce a smooth cutoff function $K\,:\, \R \to \R$ satisfying
\beq \label{def_K}
K(x) =
	\begin{cases}
	1 & \text{ if } x \leq 1/9 \,,\\
	0 & \text{ if } x \geq 2/9\,.
	\end{cases}
\eeq
Let $\caN(E_1, E_2)$ be the number of the eigenvalues in $(E_1, E_2]$, {\it i.e.},
\beq
\caN(E_1,E_2) \deq |\{ \alpha\,:\, E_1 < \wt \mu_{\alpha} \leq E_2\}|\,,\nonumber
\eeq 
and define the density of states in the interval $(E_1,E_2]$ by
\beq
\frn(E_1,E_2) \deq \frac{1}{N} \caN(E_1,E_2)\,.\nonumber
\eeq
In order to estimate $\p (\wt \mu_1 \leq E)$, we consider the following approximations:
\begin{align}\label{edge_approx}
\p (\wt \mu_1 \leq E) = \E K( \caN (E, \infty) )& \simeq \E K( \caN (E, E_*) )\simeq \E K \left( N \int_E^{E_*} \im m_{\wt Q}(y + \ii \eta) \,\dd y \right)\,,
\end{align}
with $\eta\sim N^{-2/3-\epsilon'}$, for some small $\epsilon'>0$. The first approximation in \eqref{edge_approx} follows from Lemma~\ref{lemma local law}, the rigidity of the eigenvalues, and the second from
\beq
\caN (E, E_*) = \Tr \chi_E (H) \simeq \Tr \chi_E * \theta_{\eta} (H) = \frac{1}{\pi} N \int_E^{E_*} \im m_{\wt Q}(y + \ii \eta) \,\dd y\,.\nonumber
\eeq
The following lemma shows that the approximations in~\eqref{edge_approx} indeed hold.
\begin{lemma} \label{lem approx}
Suppose that $E$ satisfies \eqref{def_E}. For $\epsilon>0$, let $\ell \deq \frac{1}{2} N^{-2/3 - \epsilon}$ and $\eta \deq N^{-2/3 - 9\epsilon}$. Recall that $K$ is a smooth function satisfying \eqref{def_K}. Then, for any sufficiently small $\epsilon > 0$ and any (large) $D > 0$, we have
\begin{align*}
\Tr \left( \chi_{E + \ell} * \theta_{\eta} (H) \right) - N^{-\epsilon} \leq \frn (E, \infty) \leq \Tr \left( \chi_{E - \ell} * \theta_{\eta} (H) \right) + N^{-\epsilon}
\end{align*}
and
\begin{align*}
\E K \left( \Tr \left( \chi_{E - \ell} * \theta_{\eta} (H) \right) \right) \leq \p (\wt \mu_1 \leq E) \leq \E K \left( \Tr \left( \chi_{E + \ell} * \theta_{\eta} (H) \right) \right) + N^{-D}\,,
\end{align*}
for any sufficiently large $N \geq N_0 (\epsilon, D)$.
\end{lemma}

\begin{proof}
We may follow the proof of Corollary 6.2 of \cite{PY}. Note that the estimates on $|m_{\wt Q} (E + \ii \ell) - \widehat m_{\fc} (E + \ii \ell)|$ and $\im \widehat m_{\fc} (E - \kappa + \ii \ell)$, which replace similar estimates with respect to $m_c$ in the proof of Corollary 6.2 in \cite{PY}, are already proved in Lemma~\ref{lem:square root for hat mfc} and Lemma~\ref{lemma local law}.
\end{proof}

\section{Green Function Comparison and Proof of the Main Result}

Having established Lemma~\ref{lem approx}, the proof of Theorem~\ref{thm main} directly follows from our main technical result: the Green function comparison theorem at the edge, Proposition~\ref{prop green} below. It compares the expectations of functions of the averaged Green functions of $\wt Q$ and $X^* X$. More precisely, we let
\begin{align}\label{le W definition}
W \deq \sqrt d \, (1 + \sqrt d)^{-4/3} X^* X\,,
\end{align}
and introduce
\begin{align*}
m_W (z) \deq \frac{1}{N} \mathrm{Tr } (W - z)^{-1}\,,\qquad\qquad(z\in\C^+)\,.
\end{align*}
It is well known that the distribution of the rescaled largest eigenvalue of $W$ converges to the Tracy-Widom distribution; see \cite{PY}. 

 Our main technical result is as follows. Recall that we write $L_+=\gamma_0 E_+$, with $E_+$ given in~\eqref{E+} and with $\gamma_0$ given in~\eqref{le gamma0}.

\begin{proposition}[Green function comparison] \label{prop green}
Let $\epsilon>0$ and set $\eta = N^{-2/3 - \epsilon}$. Denote by $M_+$ the upper edge of the Marchenko-Pastur law~$\rho_{\MP}$ for~$W = \sqrt d (1 + \sqrt d)^{-4/3} X^* X$. Let $E_1, E_2\in\R$ satisfy $E_1 < E_2$ and
\begin{align}\label{E1 and E2}
|E_1| \,, |E_2| \leq N^{-2/3 + \epsilon}\,.
\end{align}
Let $F : \R \to \R$ be a smooth function satisfying
\beq \label{F bound}
\max_x |F^{(\ell)}(x)| (|x|+1)^{-C} \leq C\,,\qquad \qquad \ell=1,2, 3, 4\,.
\eeq

Then, there exists a constant $\phi > 0$ such that, for any sufficiently large~$N$ and for any sufficiently small $\epsilon > 0$, we have
\begin{multline} \label{green_comp}
\bigg| \E F \bigg( N \int_{E_1}^{E_2} \im m_{\wt Q} (x + L_+ + \ii \eta) \,\dd x \bigg) - \E F \bigg( N \int_{E_1}^{E_2} \im m_W (x + M_+ + \ii \eta)\, \dd x \bigg) \bigg| \leq N^{-\phi}\,.
\end{multline}
\end{proposition}
We outline the proof of Proposition \ref{prop green} in the Appendices~\ref{sec:expansion of GiaGai},~\ref{sec:optical theorems} and~\ref{sec:proof of getting rid of the rest}.

\begin{remark}
Proposition \ref{prop green} can be extended as follows: Let $\epsilon>0$ and set $\eta = N^{-2/3 - \epsilon}$. Let $E_0, E_1, \cdots, E_k \in\R$ satisfy $E_1 < E_2 < \cdots < E_k$ and
\begin{align*}
|E_0| \leq N^{-2/3 + \epsilon}\,, \qquad |E_1 - E_0| \leq N^{-2/3 + \epsilon}\,, \quad \cdots \,, \quad |E_k - E_0| \leq N^{-2/3 + \epsilon}\,.\nonumber
\end{align*}

Let $F : \R^k \to \R$ be a smooth function satisfying
\beq
\max_x |F^{(\ell)}(x)| (|x|+1)^{-C} \leq C\,,\qquad \qquad \ell=1, 2, 3, 4\,.\nonumber
\eeq
Then, there exists a constant $\phi > 0$ such that, for any sufficiently large $N$ and for any sufficiently small $\epsilon > 0$, we have
\begin{multline*}
\Bigg| \E F \Bigg(\bigg( N \int_{E_i}^{E_0} \im m_{\wt Q} (x + L_+ + \ii \eta)\, \dd x \bigg)_{1\le i\le k} \Bigg) \\
- \E F \Bigg( \bigg( N \int_{E_i}^{E_0} \im m_W (x + M_+ + \ii \eta) \,\dd x\bigg)_{1\le i\le k}\Bigg) \Bigg| \leq N^{-\phi}\,. 
\end{multline*}
The proof of this statement is similar to that of Proposition~\ref{prop green} and will be omitted.

 Assuming the validity of Proposition~\ref{prop green}, we now prove our main results.

\end{remark}

\begin{proof}[Proof of Theorem \ref{thm main}]
We follow the proof of Theorem 1.10 of \cite{PY}. Let~$\mu_1^W$ be the largest eigenvalue of~$W$ (see~\eqref{le W definition}) and denote by $M_+$ the upper edge of the rescaled Marchenko-Pastur law $\rho_{\MP}$. We notice that the distribution of~$N^{2/3} (\mu_1^W - M_+)$ converges to the Tracy-Widom law $F_1$. (See \cite{So,Pe1,FS,PY}.) Thus, in order to prove~\eqref{eq:main}, it suffices to show that
\begin{align} \label{edge}
\p\, [N^{2/3} (\mu_1^{W} - M_+) \leq s] - N^{-\phi} &\leq \p \,[N^{2/3} \big( \wt \mu_1 - L_+ \big) \leq s] \nonumber\\ &\leq \p\, [N^{2/3} (\mu_1^{W} - M_+) \leq s] + N^{-\phi}\,,
\end{align}
for some $\phi > 0$.

Fix an $s$ satisfying $|s| \prec 1$, and let $E \deq L_+ + s N^{-2/3}$. Let~$\ell \deq \frac{1}{2} N^{-2/3 - \epsilon}$ and~$\eta\deq N^{-2/3 - 9\epsilon}$. For any sufficiently small $\epsilon > 0$, we have from Lemma~\ref{lem approx} that
\beq
\p (\wt \mu_1 \leq E) \geq \E \left[ K \left( \Tr \left( \chi_{E - \ell} * \theta_{\eta} (H) \right) \right) \right]\,.
\eeq
From Proposition \ref{prop green}, we find that
\beq
\E \left[ K \left( \Tr \left( \chi_{E - \ell} * \theta_{\eta} (H) \right) \right) \right] \geq \E \left[ K \left( \Tr \left( \chi_{E - (L_+ - M_+) - \ell} * \theta_{\eta} (W) \right) \right) \right] - N^{-\phi}\,,\nonumber
\eeq
for some $\phi > 0$. Finally, we have from Corollary 6.2 of \cite{PY} that
\beq
\E \left[ K \left( \Tr \left( \chi_{E - (L_+ - M_+) - \ell} * \theta_{\eta} (W) \right) \right) \right] \geq \p \left(\mu_1^{W} \leq E - (L_+ - M_+) \right) - N^{-\phi}\,.\nonumber
\eeq
Altogether, we have shown that
\beq
\p (\wt \mu_1 \leq E) \geq \p \left(\mu_1^W \leq E - (L_+ - M_+) \right) - 2 N^{-\phi}\,,\nonumber
\eeq
which proves the first inequality of \eqref{edge}. The second inequality can be proved similarly. 
\end{proof}

\begin{proof}[Proof of Corollary \ref{cor main}]
Let $U$ be an $M \times M$ orthogonal matrix that diagonalizes $\Sigma$, {\it i.e.}, there exists an $M \times M$ real diagonal matrix $D$ such that $\Sigma = U^* D U$. Then, $UX$ is a real random matrix with Gaussian entries, satisfying the assumptions in Definition~\ref{assumption sample}. Thus, applying Theorem \ref{thm main} with $X^* \Sigma X = (UX)^* D (UX)$, we get the desired result.
\end{proof}

\section{Linearization of $\widetilde Q$} \label{linear}

In this section, we recall a well-known formalism that simplifies the computations in the proof of Proposition~\ref{prop green} considerably. Instead of working with the product matrices $\widetilde Q=X\adj T X$ or $T^{1/2}XX\adj T^{1/2}$, we may ``linearize'' the problem by introducing an $(N+M)\times(N+M)$ matrix~$H$, whose respective entries are either $(x_{\alpha a})$, $(t^{-1}_\alpha)$,~$z$ or simply zero. The inverse of $H$ is then related to the Green function of $X\adj TX$, respectively of $T^{1/2}XX\adj T^{1/2}$, through Schur's complement formula or the Feshbach map. For similar applications in random matrix theory see, {\it e.g.},~\cite{A1,Gir}.

The linearization of $\widetilde Q$ is established in Subsection~\ref{section: schur complement}. In the Subsections~\ref{sec:Green function, minors and partial expectations},~\ref{sec:Green function identities} and~\ref{sec: local law for H at the edge}, we collect useful technical results on the inverse of~$H$.
\subsection{Schur complement}\label{section: schur complement}

Suppose that $X$ and $\Sigma$ satisfy the assumptions in Theorem \ref{thm main}. Let $z$ be as in the previous section. We define an $(N+M)\times (N+M)$ matrix $H$ as follows. Let $P$ be the projection on the first $N$ coordinates in $\C^{N+M}$ and set $\overline P\deq\lone-P$. Then, we write
\begin{align}\label{le H}
 H=PHP+PH\overline{P}+\overline PH P+\overline{P}H\overline P\,,
\end{align}
where
\begin{align*}
 PHP\deq-z\lone\,,\qquad PH\overline{P}\deq X\adj\,,\qquad \overline{P}HP\deq X\,,\qquad\overline{P} H \overline P\deq -T^{-1}\,,
\end{align*}
with $T = \gamma_0 \Sigma$. Note that $H(z)$ is invertible for $z\in\C^+$: Assuming that $\bsv$, $\bsv\not=0$, is in the kernel of~$H(z)$, $\im z>0$, and writing $\bsv_P\deq P\bsv$ and $\bsv_{\overline{P}}\deq \overline{P}\bsv$, we must have
\begin{align*}
-z \bsv_P + X^* \bsv_{\overline{P}} = 0\,, \qquad\qquad X \bsv_P - T^{-1} \bsv_{\overline{P}} = 0\,. 
\end{align*}
Thus,
\begin{align*}
X X^* \bsv_{\overline{P}} / z = T^{-1} \bsv_{\overline{P}}\,.
\end{align*}
Hence, taking the inner product with $\bsv_{\overline{P}}$, we find that the left side is not real while the right side is. We thus get a contradiction allowing us to conclude that $\bsv=0$.

We define the ``Green function'', $G$, of $H\equiv H(z)$ by
\begin{align}\label{le G}
 G(z)\deq H(z)^{-1}\,,\qquad\qquad (z\in\C^+)\,,
\end{align}
and the averages, $m$ and $\widetilde m$, of $G$ by
\begin{align}
 m(z)\deq \frac{1}{N}\sum_{a=1}^N G_{aa}(z)\,,\qquad \widetilde m(z)\deq\frac{1}{M}\sum_{\alpha=N+1}^{M+N}G_{\alpha\alpha}(z)\,,\qquad\quad (z\in\C^+)\,.
\end{align}
Note that by Schur's complement formula we have
\begin{align}\label{good resolvent}
 PG(z)P&=\frac{1}{PH(z)P-PH(z)\overline P\frac{1}{\overline P H(z)\overline P}\overline PH(z) P} =\frac{1}{-z P+X\adj T X}\,,
\end{align}
so that
$$
G_{ab}(z) = \left[ (\wt Q-z)^{-1} \right]_{ab}\,,
$$
for any $a, b \in \llbracket 1, N \rrbracket$. In particular,
$$
m(z) = m_{\wt Q}(z)\,.
$$
Also note that
\begin{align}\label{bad resolvent}
{z^{-1}}\overline PG(z)\overline P=\frac{{ z^{-1}}}{\overline{P}H\overline{P}-\overline{P}HP\frac{1}{PHP}P H\overline{P}}=\frac{1}{-zT^{-1}-XX\adj}\,.
\end{align}

In the following we use lowercase Roman letters for indices in $\llbracket 1,N\rrbracket$, Greek letters for indices in $\llbracket N+1,M+N\rrbracket$ and uppercase Roman letters for indices in $\llbracket 1,N+M\rrbracket$.

\subsection{Green function, minors and partial expectations}\label{sec:Green function, minors and partial expectations}
Recall the definitions of the $(N+M)\times (N+M)$ matrix $H\equiv H(z)$  in~\eqref{le H} and of the Green function $G$ in~\eqref{le G}.

Let~$\T\subset \llbracket1, N+M \rrbracket$. We then define $H^{(\T)}$ as the $(N+M-|\T|)\times(N+M-|\T|)$ minor of~$H$ obtained by removing all columns and rows of~$H$ indexed by~$\T$. We do not change the names of the indices of~$H$ when defining~$H^{(\T)}$. More specifically, we define an operation $\pi_{\mathrm A}$, $\mathrm A\in \llbracket1, N+M \rrbracket$, on the probability space by
\begin{align*}
 (\pi_\mathrm A(H))_{\mathrm {BC}}\deq\lone(\mathrm  B\not=\mathrm A)\lone(\mathrm C\not=\mathrm A)h_{\mathrm {BC}}\,.
\end{align*}
Then, for $\T\subset \llbracket1, N+M \rrbracket$, we set $\pi_{\T}\deq\prod_{\mathrm A\in\T}\pi_{\mathrm A}$ and define
\begin{align*}
 H^{(\T)}\deq((\pi_{\T}(H)_{\mathrm {BC}})_{\mathrm{B},\mathrm C\not\in\T}\,.
\end{align*}

The Green functions $G^{(\T)}$, are defined in an obvious way using~$H^{(\T)}$.  Moreover, we use the shorthand notations
\begin{align*}
 \sum_{a}^{(\T)}\deq\sum_{\substack{a=1\\a\not\in\T}}^N\,\,, \quad\sum_{a\not=b}^{(\T)}\deq\sum_{\substack{a=1,\, b=1\\ a\not=b\,,\,a,b\not\in\T}}^N\,,\quad  \sum_{\alpha}^{(\T)}\deq\sum_{\substack{\alpha=N+1\\ \alpha\not\in\T}}^{N+M}\,,\quad\sum_{\alpha\not=\beta}^{(\T)}\deq\sum_{\substack{\alpha=N+1,\, \beta=N+1\\ \alpha\not=\beta\,,\,\alpha,\beta\not\in\T}}^{N+M}\,\,,
\end{align*}
and abbreviate $(\mathrm{A})=(\{\mathrm{A}\})$, $(\T \mathrm{A})=(\T\cup\{\mathrm{A}\})$. In Green function entries $(G_{\mathrm{AB}}^{(\T)})$ we refer to $\{\mathrm{A},\mathrm{B}\}$ as lower indices and to $\T$ as upper indices.

We further set
\begin{align*}
 m^{(\T)}\deq\frac{1}{N}\sum_{a}^{(\T)}G_{aa}^{(\T)}\,,\qquad\qquad \widetilde m^{(\T)}\deq\frac{1}{M}\sum_{\alpha}^{(\T)}G_{\alpha\alpha}^{(\T)}\,.
\end{align*}
Note that we use the normalizations $N^{-1}$ and $M^{-1}$ here since they are more convenient in computations.

Finally, we denote by $\E_{a}$, $\E_{\alpha}$ the partial expectation with respect
to the variables $ (x_{\alpha a} )_{\alpha=M+1}^{N+M}$, respectively $(x_{\alpha a})_{a=1}^{ N}$.

\subsection{Green function identities}\label{sec:Green function identities}
The next lemma collects the main identities between the matrix elements of~$G$ and~$G^{(\T)}$. 

\begin{lemma}
Let $G\equiv G(z)$, $z\in\C^+$, be defined in~\eqref{le G}. Assume that the matrix $T$ is diagonal. Then, for $a,b\in\llbracket 1,N\rrbracket$, $\alpha,\beta\in\llbracket N+1,N+M\rrbracket$, $\mathrm{A},\mathrm{B},\mathrm {C}\in\llbracket 1,N+M\rrbracket$, the following identities hold: 
\begin{itemize}
 \item[-] {\it Schur complement/Feshbach formula:}\label{feshbach} For any~$a$ and~$\alpha$,
\begin{align}\label{schur} 
 G_{aa}&=\frac{1}{z-\sum_{\alpha,\beta}{x_{\alpha a} G_{\alpha\beta}^{(a)}}x_{\beta a}}\,,\nonumber\\  G_{\alpha\alpha}&=\frac{1}{-(T^{-1})_{\alpha\alpha}-\sum_{a,b}x_{\alpha a}G_{ab}^{(\alpha)}x_{\alpha b}}\,.
\end{align}

\item[-] For $a\not=b$,
\begin{align}\label{onesided}
 G_{ab}=-G_{aa}\sum_{\alpha}x_{\alpha a}G_{\alpha b}^{(a)}=-G_{bb}\sum_{\beta} G_{a\beta}^{(b)} x_{\beta b}\,.
\end{align}

\item[-] For $\alpha\not=\beta$,
\begin{align}\label{onesided greek}
 G_{\alpha\beta}=-G_{\alpha\alpha}\sum_{a}x_{\alpha a}G_{a \beta}^{(\alpha)}=-G_{\beta\beta}\sum_{b} G_{b\alpha}^{(\beta)} x_{\beta b}\,.
\end{align}

\item[-] For any $a$ and $\alpha$,
\begin{align}\label{onesided mixed}
 G_{a\alpha}=-G_{aa}\sum_{\beta}x_{\beta a}G_{\beta \alpha}^{(a)}=-G_{\alpha\alpha}\sum_{b} G_{ab}^{(\alpha)} x_{\alpha b}\,.
\end{align}

\item[-] For $\mathrm{A},\mathrm{B}\not=\mathrm{C}$,
\begin{align}\label{basic resolvent}
 G_{\mathrm{AB}}=G_{\mathrm{AB}}^{(\mathrm{C})}+\frac{G_{\mathrm{AC}}G_{\mathrm{CB}}}{G_{\mathrm{CC}}}\,.
\end{align}

\item[-] {\it Ward identity:} For any $a$,
\begin{align}\label{Ward}
 \sum_{b}|G_{ab}|^2=\frac{\im G_{aa}}{\eta}\,.
\end{align}

\end{itemize}
\end{lemma}
For a proof we refer to, {\it e.g.},~\cite{EKYY1}.

\subsection{Local law for $H$ at the edge}\label{sec: local law for H at the edge}

Consider two families of random variables $(X_i)$ and $(Y_i)$, with $i\in\llbracket 1,N\rrbracket$,  satisfying
\begin{align}\label{assumption for large deviation}
\E Z_i = 0\,, \qquad\quad \E|Z_i|^2 = 1 \,,\qquad\quad \E |Z_i|^p\le c_p\,,\qquad\quad(p\ge 3)\,,
\end{align}
$Z_i=X_i,Y_i$, for all $p\in\N$ and some constants $c_p$, uniformly in $i\in\llbracket 1,N\rrbracket$.  The following lemma, taken from~\cite{EKYY4}, provides useful large deviation estimates.
\begin{lemma}\label{lemma large deviation estimates}
 Let $(X_i)$ and $(Y_i)$ be independent families of random variables and let $(a_{ij})$ and $(b_i)$, $i, j \in\llbracket 1, N\rrbracket$, be families of complex numbers. Suppose that all entries $(X_i)$ and $(Y_i)$
are independent and satisfy~\eqref{assumption for large deviation}. Then we have the bounds:
\begin{align}
\left|\sum_i b_iX_i\right|&\prec \left(\sum_i |b_i|^2\right)^{1/2}\,,\label{LDE 1}\\
\left|\sum_{i}\sum_{j}a_{ij} X_iY_j\right|&\prec\left(\sum_{i,j}|a_{ij}|^2 \right)^{1/2}\label{LDE 2}\,,\\
\left|\sum_{i}\sum_{j}a_{ij} X_iX_j-\sum_{i}a_{ii}\right|&\prec\left(\sum_{i,j}|a_{ij}|^2 \right)^{1/2}\label{LDE 3}\,.
\end{align}
If the coefficients $a_{ij}$ and $b_i$ depend on an additional parameter~$u$, then all of these estimates are uniform
in~$u$, {\it i.e.,} the threshold $N_0 = N_0 (\epsilon, D)$ in the definition of~$\prec$ depends only on the family~$(c_p)$ from~\eqref{assumption for large deviation}; in particular,~$N_0$ does not depend on~$u$.
\end{lemma}

From the large deviation estimates in Lemma~\ref{lemma large deviation estimates} and the local law in Lemma~\ref{lemma local law}, we obtain the following estimates.
\begin{lemma}\label{local law Greek}
Let $G \equiv G(z) $, $z\in\C^+$, be defined in~\eqref{le G}. Suppose that~$T$ is diagonal, {\it i.e.},  $T=\diag(t_\alpha)$. Then, under Assumption \ref{assumptions for local MP law}, the Green function~$G$ satisfies the following bounds uniformly in~$z$ on~$\caD(\epsilon,c)$ (with $\epsilon, c>0$ as in Lemma~\ref{lemma local law}):
\begin{itemize}
\item[$i$.] For any $\alpha \in \llbracket N+1, N+M \rrbracket$,
\begin{align}
|G_{\alpha\alpha}(z)| \prec 1\,, \qquad\qquad \im G_{\alpha\alpha}(z) \prec \Psi\,.
\end{align}

\item[$ii$.] For any $a \in \llbracket 1, N \rrbracket$ and $\alpha \in \llbracket N+1, N+M \rrbracket$,
\begin{align}
|G_{a \alpha}(z)| \prec \Psi\,.
\end{align}

\item[$iii$.] For any $\alpha, \beta \in \llbracket N+1, N+M \rrbracket$ with $\alpha \neq \beta$,
\begin{align}
|G_{\alpha \beta}(z)| \prec \Psi\,.
\end{align}

\end{itemize}

\end{lemma}

\begin{proof}
From Schur's complement formula~\eqref{schur}, we obtain that
\begin{align*}
\frac{1}{G_{\alpha\alpha}} = -t_{\alpha}^{-1} - \sum_{k, l} x_{\alpha k} G^{(\alpha)}_{kl} x_{\alpha l}\,.
\end{align*}
Further, from the large deviation estimate~\eqref{LDE 3} and the Ward identity~\eqref{Ward}, we find
\begin{align}\label{bound large deviation with greek}
\left| m^{(\alpha)} - \sum_{k, l} x_{\alpha k} G^{(\alpha)}_{kl} x_{\alpha l} \right| \prec 
\sqrt{ \frac{\im m^{(\alpha)}}{N\eta}} \prec \Psi\,,
\end{align}
uniformly in $z\in\caD(\epsilon,c)$.

We next claim that $|m - m^{(\alpha)}| \prec \Psi$. To prove this claim, we first let
$$
\wt \caQ \deq T^{1/2} X X^* T^{1/2}\,.
$$
We notice that the averaged Green function $m$ can be written in terms of $\wt\caQ$ by
\begin{align*}
m(z) = m_{\wt \caQ} (z) = \frac{1}{N} \left( \Tr (\wt \caQ - z )^{-1} + \frac{N-M}{z} \right)\,.
\end{align*}
Next, we consider the minor $\wt\caQ^{(\alpha)}$, which is obtained by removing all columns and rows of $\wt\caQ$ indexed by~$\alpha$. Then,
$$
m^{(\alpha)}(z) = \frac{1}{N} \left( \Tr (\wt\caQ^{(\alpha)} - z )^{-1} + \frac{N-M+1}{z} \right)\,.
$$
By Cauchy's eigenvalue interlacing property, we get
$$
\big| \Tr (\wt\caQ - z )^{-1} - \Tr (\wt\caQ^{(\alpha)} - z )^{-1} \big| \leq C \eta^{-1}\,.
$$
(See the proof of Lemma 5.1 in \cite{BPZ2}.) This proves the desired claim.

Since $|\widehat \xi_+ + m| \prec \Psi$ and $t_{\alpha}^{-1} \ge \widehat\gamma_0(1+c) \widehat \xi_+$ for some $c > 0$ (see~\eqref{Sigma assumption}), we get $1\prec|G_{\alpha\alpha}|^{-1}$, hence $|G_{\alpha\alpha}| \prec 1$. Moreover, using once more Schur's complement formula~\eqref{schur}, we have
\begin{align*}
|\im G_{\alpha\alpha}| \prec \left| \im \sum_{k, l} x_{\alpha k} G^{(\alpha)}_{kl} x_{\alpha l} \right| \prec \im m^{(\alpha)} + \left| m^{(\alpha)} - \sum_{k, l} x_{\alpha k} G^{(\alpha)}_{kl} x_{\alpha l} \right| \prec \Psi\,,
\end{align*}
uniformly in $\caD(\epsilon,c)$, where we used~\eqref{bound large deviation with greek} and that $\im m^{(\alpha)}\prec \im m_{\fc}\sim \sqrt{\kappa+\eta}$.
This proves statement~$i$.

From the Green function identity~\eqref{onesided mixed} and statement~	$i$, we have
\begin{align*}
|G_{a \alpha}| = \left| G_{\alpha\alpha} \sum_k x_{\alpha k} G^{(\alpha)}_{ka} \right| \prec \left| \sum_k x_{\alpha k} G^{(\alpha)}_{ka} \right|\,,
\end{align*}
where we used the local law of Lemma~\ref{lemma local law}. Thus, applying the large deviation estimate~\eqref{LDE 1} and the local law $|G^{(\alpha)}_{ka}| \prec \Psi$, we get
\begin{align*}
|G_{a \alpha}| \prec \left(\frac{1}{N} \sum_k |G^{(\alpha)}_{ka}|^2 \right)^{1/2} \prec \Psi\,,
\end{align*}
uniformly in $z\in\caD(\epsilon,c)$, which proves statement~$ii$ of the lemma.

Similarly, we have from the Green function identity~\eqref{onesided greek} that
\begin{align*}
|G_{\alpha\beta}| = \left| G_{\alpha\alpha} \sum_k x_{\alpha k} G^{(\alpha)}_{k \beta} \right| \prec \left| \sum_k x_{\alpha k} G^{(\alpha)}_{k \beta} \right| \prec \left(\frac{1}{N} \sum_k |G^{(\alpha)}_{k \beta}|^2 \right)^{1/2} \prec \Psi\,,
\end{align*}
where we used $|G_{k\beta}|\prec \Psi$ to get the last inequality. This proves statement~$iii$ of the lemma.
\end{proof}

We conclude this section by giving estimates on expectations of monomials of Green functions entries.
\begin{lemma}
Let $P\equiv P(z)$ be a monomial in the Green function entries $(G_{AB}(z))$, with $z\in\caD(\epsilon,c)$, for some $\epsilon,c>0$. Then, there exists a universal constant $C$, such that
\begin{align}
 \E|P(z)|^2\prec N^{Cn}\,,
\end{align}
where $n$ is the degree of $P$. In particular, if $|P(z)|\prec \Psi(z)^k$, uniformly in $\caD(\epsilon,c)$, then $\E|P(z)|\prec\Psi(z)^k$, uniformly in $\caD(\epsilon,c)$. (See the paragraph after Definition~\ref{definition of stochastic domination}.)

Moreover, the same conclusions hold with $G^{(\T)}$ replacing $G$ for any $\T$.
\end{lemma}
\begin{proof}
First, we note that $|G_{ab}|\le\frac{1}{\eta}$, $a,b\in\llbracket 1,N\rrbracket$, as follows from the self-adjointness of $X\adj TX$ and the spectral calculus.

Second, to bound $|G_{\alpha\beta}|$, $\alpha,\beta\in\llbracket N+1,M+N\rrbracket$, we recall that $T^{-1}$ is a strictly positive operator by Assumption~\ref{assumptions for local MP law}. Thus
\begin{align*} 
\im \langle \bsv, (zT^{-1}- XX\adj) \bsv\rangle = \eta \langle \bsv, T^{-1}\bsv\rangle\ge c\eta\, \|\bsv\|^2\,,\qquad\quad \forall \bsv\in\C^{M}\,,
 \end{align*}
 for some $c>0$ independent of $\bsv$, where $\langle \cdot,\cdot\rangle$ denotes the canonical inner product in $\C^M$. Since ${ z^{-1}} \overline P G(z)\overline P=({-zT^{-1}+XX\adj})^{-1}$, $|z|>0$, we get $|G_{\alpha\beta}|\le \frac{C { |z|}}{\eta}$. 

 Third, to bound~$\E|G_{a \alpha}|^p$, $a\in\llbracket 1,N\rrbracket$, $\alpha\in\llbracket N+1,N+M\rrbracket$, $p\ge 0$, we note that by~\eqref{onesided mixed} we have
 \begin{align}
  |G_{a \alpha}|=|G_{aa}|\sum_{\beta}|x_{\beta a} G_{\beta\alpha}^{(a)}|\le \frac{C { |z|}}{\eta^2}N|x_{\beta a}|\,,
 \end{align}
by the estimates above. From the moment bounds in Assumption~\ref{le subexponential decay}, we then conclude that $\E|G_{\alpha a}|^p\le C^p N^{cp}$, where we also used that $\eta\gg N^{-1}$,  $0<|z| < C$ by assumption.

The lemma now easily follows from H\"{o}lder's inequality.
\end{proof}

In the rest of the paper, we prove Proposition \ref{prop green} with the formalism outlined in this section. The actual calculation will be done for the simple case $F' \equiv 1$; the proof for general $F'$ is basically the same, though the computations are much longer for this case. The details for $F'\not\equiv 1$ can be found in~\cite{LSY}.

\section{Green Function Flow} \label{G flow}

The key idea of our proof of Proposition~\ref{prop green} is similar to the one of the proof of Proposition~5.2 in~\cite{LSY} for deformed Wigner matrices: We consider a continuous interpolation between the sample covariance matrices $\wt Q$ and $W$ by introducing a time evolution that deforms~$T$ continuously to the identity. We then track the associated flow of the Green function for sufficiently long time. The outcome is an estimate on the time derivative of the Green function which is sufficiently accurate to prove Proposition~\ref{prop green}.

\subsection{Preliminaries}

Suppose that $T=\gamma_0\Sigma$ is diagonal, {\it i.e.}, $T=\diag(t_\alpha)$. We interpolate between $\Sigma=\diag(\sigma_\alpha)$ and the identity matrix $\lone$ by introducing the time evolution $t\mapsto(\sigma_{\alpha}(t))$ defined by
\begin{align}\label{le time evolution of Sigma}
\frac{1}{\sigma_{\alpha}(t)} = \e{-t} \frac{1}{\sigma_{\alpha}(0)} + (1-\e{-t})\,,\quad \Sigma(t)=\diag(\sigma_\alpha(t))\,,\qquad (t\ge 0)\,.
\end{align}
We let $z \equiv z(t)$ be time-dependent and define the $(N+M) \times (N+M)$ matrix $H(t) \equiv H( z,t)$ by
\begin{align*}
PH(t) P \deq -z(t)\,, \qquad  \overline P H(t) \overline P \deq -T^{-1}(t)\,,
\end{align*}
and 
\begin{align*}
 PH(t) \overline P \deq X^*\,, \qquad\overline P H(t) P \deq X\,,
\end{align*}
with $T(t)=\gamma(t) \Sigma(t)$, $T(0)=\gamma_0\Sigma$, for some time-dependent scaling factor $\gamma(t)\in\R$ (see~\eqref{le gamma} below for the definition of $\gamma(t)$). We also let
\begin{align}
G(z,t) \deq H(z,t)^{-1}\,, \qquad m(z,t) = \frac{1}{N} \Tr G(z,t)\,,\quad\qquad ( z\in\C^+)\,.
\end{align}

From~\eqref{xi+}, it is natural to let $\xi_+(t)$ be the largest solution to
\begin{align}\label{le xi}
\frac{1}{M} \sum_{\alpha} \left( \frac{\sigma_{\alpha}(t) \xi_+(t)}{1 - \sigma_{\alpha}(t) \xi_+(t)} \right)^2 = d\,,
\end{align}
with $\xi_+(0) = \xi_+$. We then choose the scaling factor $\gamma \equiv \gamma(t)$ to be given by
\begin{align}\label{le gamma}
\gamma(t) = \left( \frac{1}{N} \sum_{\alpha} \left( \frac{\sigma_{\alpha}(t)}{1 - \sigma_{\alpha}(t) \xi_+(t)} \right)^3 + \big( \xi_+(t) \big)^{-3} \right)^{-1/3}\,,
\end{align}
with $\gamma_0=\gamma(0)$, and we also introduce
\begin{align}\label{le tau}
\tau \equiv \tau(t) \deq \frac{\xi_+(t)}{\gamma(t)}\,.
\end{align}
For simplicity, we often omit the $t$-dependence in the notation for $T(t)$, $\gamma(t)$ and $\tau(t)$ in the following.

Note that we have from~\eqref{le xi},~\eqref{le gamma} and~\eqref{le tau} that
\begin{align} \label{tau 2}
\frac{1}{N} \sum_{\alpha} \left( \frac{1}{t_{\alpha}^{-1} - \tau} \right)^2 = \frac{1}{\tau^2}\,,\qquad\qquad\frac{1}{N} \sum_{\alpha} \left( \frac{1}{t_{\alpha}^{-1} - \tau} \right)^3 + \frac{1}{\tau^3} = 1\,.
\end{align}
In the following, we refer to the identities in~\eqref{tau 2} as ``sum rules''.

We now consider the evolution of the Green function $G\equiv G(t)$ under the evolution governed by~\eqref{le time evolution of Sigma}. For the diagonal Green function entries $G_{ii}$, $i\in\llbracket 1,N\rrbracket$, we get
\begin{align} \label{dot G_ii}
 \E\,\frac{\partial G_{ii}}{\partial t}= \dot z \sum_{a} \E\,[G_{ia}G_{ai}] - \sum_{\alpha}\frac{\partial_t t_{\alpha}}{t_\alpha^2}\E\,[G_{i\alpha}G_{\alpha i}]\,.
\end{align}

\begin{remark} \label{rho_t}
Let $\widetilde m_{\fc}(z,t)$ be the solution to
\begin{align*}
 \widetilde m_{\fc}(z,t)=\frac{1}{-z+\frac{1}{dM}\sum_{\alpha}\frac{t_{\alpha}}{t_{\alpha}\widetilde m_{\fc}(z,t)+1}}\,,\qquad\qquad (z\in\C^+\,,t\ge 0)\,.
\end{align*}
such that $ \im \widetilde m_{\fc}(z,t)\ge 0$.

Setting $\widetilde\rho_{\fc}(E,t)\deq\lim_{\eta\searrow 0}\pi^{-1}\im \widetilde m_{\fc}(E+\ii\eta,t)$, we note that the rightmost point of the support of the measure $\widetilde\rho_{\fc}(t)$, denoted by $L_+\equiv L_+(t)$, is given by $L_+=\gamma E_+$, or equivalently,
\begin{align} \label{L_+}
L_+ = \frac{1}{\tau} + \frac{1}{dM} \sum_{\alpha} \frac{t_{\alpha}}{1 - t_{\alpha} \tau} = \frac{1}{\tau} + \frac{1}{N} \sum_{\alpha} \frac{1}{t_{\alpha}^{-1} - \tau} \,.
\end{align}
In fact, the rescaling by $\gamma(t)$ assures that
$$
\wt \rho (E,t) = \frac{1}{\pi} \sqrt{L_+ - E} \, \big(1 + O(L_+ - E) \big)\,,\qquad\qquad (t\ge 0)\,,
$$
as $E \nearrow L_+$, as may be checked by an explicit computation.
\end{remark}

\subsection{Proof of Proposition~\ref{prop green}}
In this subsection, we give the proof of Proposition~\ref{prop green}, which is based on two technical lemmas, Lemma~\ref{le first lemma} and Lemma~\ref{getting rid of the rest} below. For simplicity, we choose $F'\equiv \lone$. Recall the definition of the deterministic control parameter $\Psi$ in~\eqref{le definition of the psi}.

The main ingredient of the proof of the Green function comparison theorem, Proposition~\ref{prop green}, is the estimate $\im \E\, [\partial_t G_{ii}(z)] = \caO(M \Psi^5)$, for appropriately chosen~$z$. (The naive size of $\E\,[\partial_t G_{ii}]$ is $\caO(M \Psi^2)$ as one sees from \eqref{dot G_ii}.)  Once we have established the estimate $\im \E\, [\partial_t G_{ii}(z)] = \caO(M \Psi^5)$, we can integrate it from $t=0$ to $t = 2 \log N$ to compare $\im m_{\widetilde Q}$ with $\im m|_{t = 2 \log N}$. The comparison between $\im m|_{t = 2 \log N}$ and $m_{W}$ can easily be done, since $\Sigma(t)$ is close enough to the identity at~$t = 2 \log N$.

To show that the imaginary part of~\eqref{dot G_ii} is much smaller than its naive size, we use, in a first step, the following ``decoupling'' lemma. 

\begin{lemma}\label{le first lemma}
Under the assumptions of Proposition~\ref{prop green} the following holds true. Let $z(t)\equiv z=L_+ (t) +y+\ii \eta$, with $L_+(t)$ as in~\eqref{L_+},  $y\in[-N^{-2/3+\epsilon},N^{-2/3+\epsilon}]$ and $\eta=N^{-2/3-\epsilon}$.

 Then, there are $z$-dependent random variables $X_{22}$, $X_{32}$, $X_{33}$, $X_{42}$, $X_{43}$, $X_{44}$ and $X_{44}'$, satisfying
\begin{align*}
 X_{22}=\caO(\Psi^2)\,,\quad\qquad X_{32}\,,X_{33}=\caO(\Psi^3)\,,\quad\qquad X_{42}\,,X_{43}\,,X_{44}\,,X_{44}'=\caO(\Psi^4)\,,
\end{align*}
such that
\begin{align} \label{le expansion in first lemma}
\E_{\alpha} [G_{i\alpha}G_{\alpha i}]& =\frac{1}{(t_{\alpha}^{-1} -\tau)^2} X_{22} - \frac{2}{(t_{\alpha}^{-1} -\tau)^3} X_{32} - \frac{2}{(t_{\alpha}^{-1} -\tau)^3} X_{33} + \frac{3}{(t_{\alpha}^{-1} -\tau)^4} X_{42}\nonumber\\
& \qquad  + \frac{6}{(t_{\alpha}^{-1} -\tau)^4} X_{43} + \frac{12}{(t_{\alpha}^{-1} -\tau)^4} X_{44}+ \frac{3}{(t_{\alpha}^{-1} -\tau)^4} X_{44}' + \caO(\Psi^5) \,, 
\end{align}
uniformly in $t\ge 0$. The random variables above are explicitly given by
\begin{align*}
X_{22} &= \frac{1}{N} \sum_{s} G_{is} G_{si}\,,\quad\qquad\qquad &X_{32}&=\left(m+\tau\right) \frac{1}{N} \sum_{s} G_{is} G_{si}\,,\nonumber\\
X_{33}&= \frac{1}{N^2} \sum_{r, s} G_{ir} G_{rs} G_{si}\,, & X_{42} &= \left( m+\tau\right)^2 \frac{1}{N} \sum_{s} G_{is} G_{si} \,,\nonumber\\
X_{43}&= \left(m+\tau\right) \frac{1}{N} \sum_{r, s} G_{ir} G_{rs} G_{si}\,, &X_{44} &=\frac{1}{N^3} \sum_{r, s,t} G_{ir} G_{rs} G_{st} G_{ti}\,, \nonumber\\
X_{44}' & = \frac{1}{N^3} \sum_{r,s, t} G_{is} G_{si} G_{rt} G_{tr}\,,
\end{align*}
where $G\equiv G(z(t),t)$, $m\equiv m(z(t),t)=\frac{1}{N}\sum_{s}G_{ss}(z(t),t)$ and $\tau(t)$ is defined in~\eqref{le tau}.
\end{lemma}
We refer to Lemma~\ref{le first lemma} as a ``decoupling'' lemma, since on the right side of~\eqref{le expansion in first lemma} the Greek index~$\alpha$ is, up to the error $\caO(\Psi^5)$, decoupled from the Green functions which only have Roman indices as lower indices. Lemma~\ref{le first lemma} is proven in the Appendix~\ref{sec:expansion of GiaGai} below.

Taking the time derivative of~\eqref{L_+} we get
\begin{align} \label{dot z 1}
\dot z = \dot L_+ = -\frac{\dot \tau}{\tau^2} + \frac{1}{dM} \sum_{\alpha} \frac{t_{\alpha}^2 \dot\tau}{(1 - t_{\alpha} \tau)^2} + \frac{1}{dM} \sum_{\alpha} \frac{\partial_t t_{\alpha}}{1 - t_{\alpha} \tau} + \frac{1}{dM} \sum_{\alpha} \frac{\tau t_{\alpha} (\partial_t t_{\alpha})}{(1 - t_{\alpha} \tau)^2}\,.
\end{align}
From~\eqref{le tau} we observe that the first two terms on the right side of~\eqref{dot z 1} cancel. Thus, simplifying the last two terms in~\eqref{dot z 1}, we obtain
\begin{align} \label{dot z}
\dot z = \frac{1}{dM} \sum_{\alpha} \frac{\partial_t t_{\alpha}}{(1 - t_{\alpha} \tau)^2} = \frac{1}{N} \sum_{\alpha} \frac{\partial_t t_{\alpha}}{t_{\alpha}^2}\frac{1}{(t_{\alpha}^{-1} - \tau)^2}\,.
\end{align}
Hence, plugging~\eqref{le expansion in first lemma} into~\eqref{dot G_ii} we find
\begin{align}\label{after dot z}
 \E\,\frac{\partial G_{ii}}{\partial t}&= \dot z \sum_a\E\,[G_{ia}G_{ai}] - \sum_\alpha\frac{\partial_t t_{\alpha}}{t_\alpha^2}\frac{1}{(t_\alpha^{-1}-\tau)^2}\frac{1}{N}\sum_{a}\E\,[G_{ia}G_{ai}]\nonumber\\
&\qquad+ \sum_{\alpha}\frac{\partial_t t_{\alpha}}{t_\alpha^2}\E\left[ \frac{2}{(t_{\alpha}^{-1} -\tau)^3} X_{32} + \frac{2}{(t_{\alpha}^{-1} -\tau)^3} X_{33}\right]\nonumber\\
&\qquad-\sum_{\alpha}\frac{\partial_t t_{\alpha}}{t_\alpha^2}\E\,\left[\frac{3}{(t_{\alpha}^{-1} -\tau)^4} X_{42} + \frac{6}{(t_{\alpha}^{-1} -\tau)^4} X_{43}  \right]\nonumber\\
&\qquad-\sum_{\alpha}\frac{\partial_t t_{\alpha}}{t_\alpha^2}\E\,\left[ \frac{12}{(t_{\alpha}^{-1} -\tau)^4} X_{44} + \frac{3}{(t_{\alpha}^{-1} -\tau)^4} X_{44}'\right]+\caO(M\Psi^5) \,. 
\end{align}
Note that the first two terms in~\eqref{after dot z} cancel by~\eqref{dot z} and that we have
\begin{align}\label{second after dot z}
 \E\,\frac{\partial G_{ii}}{\partial t}&= \sum_{\alpha}\frac{\partial_t t_{\alpha}}{t_\alpha^2}\E\left[ \frac{2}{(t_{\alpha}^{-1} -\tau)^3} X_{32} + \frac{2}{(t_{\alpha}^{-1} -\tau)^3} X_{33}\right]\nonumber\\
&\qquad-\sum_{\alpha}\frac{\partial_t t_{\alpha}}{t_\alpha^2}\E\,\left[\frac{3}{(t_{\alpha}^{-1} -\tau)^4} X_{42} + \frac{6}{(t_{\alpha}^{-1} -\tau)^4} X_{43}  \right]\nonumber\\
&\qquad-\sum_{\alpha}\frac{\partial_t t_{\alpha}}{t_\alpha^2}\E\,\left[ \frac{12}{(t_{\alpha}^{-1} -\tau)^4} X_{44} + \frac{3}{(t_{\alpha}^{-1} -\tau)^4} X_{44}'\right]+\caO(M\Psi^5) \,. 
\end{align}

To complete the proof of Proposition~\ref{prop green}, we are going to show that the imaginary part of the right side of~\eqref{second after dot z} is of $\caO(M\Psi^5)$ as is noted in the next lemma.
\begin{lemma}\label{getting rid of the rest}
 Under the assumptions of Proposition~\ref{prop green} with the notation of Lemma~\ref{le first lemma}, we have
\begin{align}\label{le eq getting rid of the rest}
&\sum_{\alpha}\frac{\partial_t t_{\alpha}}{t_\alpha^2}\im \E\left[ \frac{2}{(t_{\alpha}^{-1} -\tau)^3} X_{32} + \frac{2}{(t_{\alpha}^{-1} -\tau)^3} X_{33}\right]-\sum_{\alpha}\frac{\partial_t t_{\alpha}}{t_\alpha^2}\im \E\,\left[\frac{3}{(t_{\alpha}^{-1} -\tau)^4} X_{42} \right]\nonumber\\
&\quad-\sum_{\alpha}\frac{\partial_t t_{\alpha}}{t_\alpha^2}\im \E\,\left[ \frac{6}{(t_{\alpha}^{-1} -\tau)^4} X_{43} + \frac{12}{(t_{\alpha}^{-1} -\tau)^4} X_{44} + \frac{3}{(t_{\alpha}^{-1} -\tau)^4} X_{44}' \right]=\caO(M\Psi^5)\,,
\end{align}
uniformly in $t\ge 0$.
\end{lemma}

We remark that the naive size of the right side of~\eqref{le eq getting rid of the rest} is~$\caO(M \Psi^3)$, but for our choice of $\gamma$ the terms cancel up to errors of $\caO(M\Psi^5)$. Similar to the discussion in~\cite{LSY}, the sum rules in~\eqref{tau 2} have crucial roles in this cancellation mechanism. Lemma~\ref{getting rid of the rest} is proven in the Appendix~\ref{sec:proof of getting rid of the rest}.

\begin{proof}[Proof of Proposition~\ref{prop green}]
For simplicity we choose $F' \equiv \lone$. From~\eqref{second after dot z} and Lemma~\ref{getting rid of the rest}, we find that
\begin{align} \label{G_ii derivative bound}
\E \left[ \im \frac{\partial G_{ii}}{\partial t} \right] =\caO(\Psi^2)\,.
\end{align}
Integrating both sides of \eqref{G_ii derivative bound} from $t=0$ to $t = 2 \log N$ we obtain that 
\begin{multline} \label{green estimate 1}
\bigg| \E \bigg[ N \int_{E_1}^{E_2} \im m( x + L_+ + \ii \eta)\big|_{t=0}\, \dd x \bigg] \\ - \E \bigg[ N \int_{E_1}^{E_2} \im m(x + L_+ + \ii \eta)\big|_{t=2 \log N}\, \dd x \bigg] \bigg|\leq N^{-1/3 + C'\epsilon}\,,
\end{multline}
for some constant $C' > 0$.

At $t = \infty$, we have $\sigma_{\alpha}(\infty) = 1$, for all $\alpha \in \llbracket N+1, N+M \rrbracket$, hence by definition
$$
\xi_+(\infty) = \frac{\sqrt d}{1 + \sqrt d}\,, \qquad\qquad \gamma(\infty) = \sqrt d \,(1 + \sqrt d)^{-4/3}\,.
$$
In particular,
$$
m(x + L_+ + \ii \eta)\big|_{t=\infty} = m_W (x + M_+ + \ii \eta)\,.
$$

Let $T_\frf \deq 2 \log N$. At $t = T_\frf$, we have $\sigma_{\alpha}(T_\frf) = 1 + O(N^{-2})$. Using the result at $t = \infty$, it can be easily seen that 
$$
\gamma(T_\frf) = \gamma(\infty) + O(N^{-2})\,.
$$
Similarly, we also have that $z(T_\frf) = z(\infty) + O(N^{-2})$. Thus, the matrix $H(T_\frf) - H(\infty)$ is a diagonal matrix whose entries are $O(N^{-2})$.

Using the resolvent identity
\begin{align}
G(T_\frf) - G(\infty) = -G(T_\frf) \left( H(T_\frf) - H(\infty) \right) G(\infty)\,,
\end{align}
we can now bound
\begin{align*}
|G_{ii} (T_\frf) - G_{ii} (\infty)| = \left| \sum_A -G_{iA}(T_\frf) \left( H_{AA}(T_\frf) - H_{AA}(\infty) \right) G_{Ai}(\infty) \right| \prec N^{-5/3}\,,
\end{align*}
and we thus have
\begin{multline} \label{green estimate 2}
\bigg| \E \bigg[ N \int_{E_1}^{E_2} \im m( x + L_+ + \ii \eta)\big|_{t=2 \log N}\, \dd x \bigg]\\ - \E \bigg[ N \int_{E_1}^{E_2} \im m(x + L_+ + \ii \eta)\big|_{t= \infty}\, \dd x \bigg] \bigg| \leq N^{-4/3 + C'\epsilon}\,.
\end{multline}
Since $m( x + L_+ + \ii \eta)\big|_{t=0} = m_{\wt Q} (x + L_+ + \ii \eta)$, we get the desired result from~\eqref{green estimate 1} and~\eqref{green estimate 2}.
\end{proof}

\begin{appendix}
 
\section{Proof of Lemma~6.2}\label{sec:expansion of GiaGai}
In this section we prove Lemma~\ref{le first lemma}. We start expanding $\E \,[G_{i\alpha}G_{\alpha i}]$ in the random variables indexed by the Greek index~$\alpha$. The following expansion follows closely the expansions used in~\cite{LSY}.

\begin{proof}[Proof of Lemma~\ref{le first lemma}]
 
Using the formula for $G_{i\alpha}$ in~\eqref{onesided}, {\it i.e.},
\begin{align*}
 G_{i\alpha}=-G_{\alpha\alpha} \sum_{k}x_{\alpha k} G^{(\alpha)}_{ik}\,,
\end{align*}
we expand $G_{i\alpha}G_{\alpha i}$ in the lower index~$\alpha$ as
\begin{align} \label{one-sided 2}
 G_{i\alpha}G_{\alpha i}=G_{\alpha\alpha}^2\sum_{k,l}G_{ik}^{(\alpha)}x_{\alpha k}x_{\alpha l} G_{li}^{(\alpha)}\,.
\end{align}
Note that, by Schur's complement formula~\eqref{schur},
\begin{align}\label{le what we get from schur}
 G_{\alpha\alpha}=\frac{1}{h_{\alpha\alpha}-\sum_{p, q}^{(\alpha)}h_{\alpha p}G^{(\alpha)}_{pq} h_{q \alpha}} = \frac{1}{-t_{\alpha}^{-1} - \sum_{p, q} x_{\alpha p}G^{(\alpha)}_{pq} x_{\alpha q}}\,.
\end{align}
(The use of Roman letters $p, q$ can be justified since $h_{\alpha p} = 0$ for $p \in \llbracket N+1, N+M \rrbracket$ and $p \neq \alpha$.) 

We next expand $G_{\alpha\alpha}$ around $(-t_{\alpha}^{-1}+\tau)^{-1}$. (Note that $\limsup t_{\alpha} \tau < 1$, thus $t_{\alpha}^{-1} - \tau > c > 0$ for some constant $c$ independent of $N$.) From the large deviation estimates in Lemma~\ref{lemma large deviation estimates} and the Ward identity~\eqref{Ward}, we have
\begin{align}
\left| \sum_{p, q} x_{\alpha p}G^{(\alpha)}_{pq}x_{\alpha q}+\tau \right| \prec \Psi\,.
\end{align}
Returning to~\eqref{le what we get from schur}, we thus have
\begin{multline*}
G_{\alpha\alpha}=\frac{1}{-t_{\alpha}^{-1} +\tau}+ \frac{1}{(-t_{\alpha}^{-1} +\tau)^2}\left(\sum_{p, q} x_{\alpha p}G^{(\alpha)}_{pq}x_{\alpha q}+\tau\right)\\ + \frac{1}{(-t_{\alpha}^{-1} +\tau)^3}\left(\sum_{p, q} x_{\alpha p}G^{(\alpha)}_{pq}x_{\alpha q}+\tau\right)^2
+ \caO(\Psi^3) \,,
\end{multline*}
respectively,
\begin{multline*}
G_{\alpha\alpha}^2 =\frac{1}{(t_\alpha^{-1} -\tau)^2} - \frac{2}{(t_{\alpha}^{-1} -\tau)^3} \left(\sum_{p, q} x_{\alpha p}G^{(\alpha)}_{pq}x_{\alpha q}+\tau\right)\\+  \frac{3}{(t_{\alpha}^{-1} -\tau)^4}\left(\sum_{p, q} x_{\alpha p}G^{(\alpha)}_{pq}x_{\alpha q}+\tau\right)^2 
 + \caO(\Psi^3)\,.
\end{multline*}
Hence, from the resolvent identity~\eqref{one-sided 2}, obtain the following expansion of $G_{i\alpha}G_{\alpha i}$ in the lower index~$\alpha$,
\begin{align*}
G_{i \alpha} G_{\alpha i}&= \frac{1}{(t_{\alpha}^{-1} -\tau)^2} \sum_{s, t}G_{is}^{(\alpha)}x_{\alpha s}x_{\alpha t} G_{ti}^{(\alpha)}\nonumber\\ &\quad -\frac{2}{(t_{\alpha}^{-1} -\tau)^3}\left(\sum_{p,q}^{(\alpha)}x_{\alpha p}G^{(\alpha)}_{pq}x_{\alpha q}+\tau\right) \sum_{s,t}G_{is}^{(\alpha)}x_{\alpha s}x_{\alpha t} G_{ti}^{(\alpha)} \nonumber \\
&\quad +\frac{3}{(t_{\alpha}^{-1} -\tau)^4} \left(\sum_{p, q} x_{\alpha p}G^{(\alpha)}_{pq}x_{\alpha q}+\tau\right)^2 \sum_{s,t}G_{is}^{(\alpha)}x_{\alpha s}x_{\alpha t} G_{ti}^{(\alpha)} + \caO(\Psi^5) \,.
\end{align*}
Taking the partial expectation $\E_\alpha$ we get
\begin{align} \label{1+2+3+4+5+6+7}
&\E_{\alpha}\, [G_{i\alpha}G_{\alpha i}] =\frac{1}{(t_{\alpha}^{-1} -\tau)^2} \frac{1}{N} \sum_s G_{is}^{(\alpha)} G_{si}^{(\alpha)} \nonumber \\
&\quad -\frac{2}{(t_{\alpha}^{-1} -\tau)^3} \left(m^{(\alpha)}+\tau\right) \frac{1}{N}\sum_s G_{is}^{(\alpha)}G_{si}^{(\alpha)} -\frac{4}{(t_{\alpha}^{-1} -\tau)^3} \frac{1}{N^2} \sum_{s, t}G^{(\alpha)}_{is}G^{(\alpha)}_{st}G^{(\alpha)}_{ti}\nonumber \\
&\quad+\frac{3}{(t_{\alpha}^{-1} -\tau)^4} \left(m^{(\alpha)}+\tau\right)^2 \frac{1}{N}\sum_s G_{is}^{(\alpha)}G_{si}^{(\alpha)} + \frac{12}{(t_{\alpha}^{-1} -\tau)^4} \left(m^{(\alpha)}+\tau\right) \frac{1}{N^2}\sum_{s, t} G_{is}^{(\alpha)} G_{st}^{(\alpha)} G_{ti}^{(\alpha)} \nonumber \\
&\quad +\frac{6}{(t_{\alpha}^{-1} -\tau)^4} \frac{1}{N^3} \sum_{s, p, q} G_{is}^{(\alpha)} G_{si}^{(\alpha)} G_{pq}^{(\alpha)} G_{qp}^{(\alpha)} + \frac{24}{(t_{\alpha}^{-1} -\tau)^4} \frac{1}{N^3} \sum_{s, p, q} G_{is}^{(\alpha)} G_{sp}^{(\alpha)} G_{pq}^{(\alpha)} G_{qi}^{(\alpha)}+ \caO(\Psi^5) \,.
\end{align}

In a next step, we expand~\eqref{1+2+3+4+5+6+7} in the upper index~$\alpha$ by using the resolvent formula~\eqref{basic resolvent}, {\it i.e.},
\begin{align}\label{le how to remove upper index}
 G_{is}^{(\alpha)}=G_{is}-\frac{G_{i\alpha}G_{\alpha s}}{G_{\alpha\alpha}}\,.
\end{align}
In other words, using~\eqref{le how to remove upper index}, we can remove the upper index~$\alpha$ from the Green functions entries in~\eqref{1+2+3+4+5+6+7} at the expense of higher order terms containing~$\alpha$ as a lower index in the Green function entries. We obtain for the first term in~\eqref{1+2+3+4+5+6+7} that
\begin{align} \label{11+12+13+14}
G^{(\alpha)}_{is} G^{(\alpha)}_{si} &= G_{is} G_{si} - \frac{G_{i \alpha} G_{\alpha s}}{G_{\alpha\alpha}} G_{si} - G^{(\alpha)}_{is} \frac{G_{s \alpha} G_{\alpha i}}{G_{\alpha\alpha}} \nonumber\\
&= G_{is} G_{si} - \frac{G_{i \alpha} G_{\alpha s}}{G_{\alpha\alpha}} G^{(\alpha)}_{si} - G^{(\alpha)}_{is} \frac{G_{s \alpha} G_{\alpha i}}{G_{\alpha\alpha}} - \frac{G_{i \alpha} G_{\alpha s}}{G_{\alpha\alpha}} \frac{G_{s \alpha} G_{\alpha i}}{G_{\alpha\alpha}}\,. 
\end{align}
We stop expanding the first term on the right side of~\eqref{11+12+13+14}, since it does not contain the index~$\alpha$, and we set
\begin{align} \label{X_22}
X_{22} \deq \frac{1}{N} \sum_{s} G_{is} G_{si}\,.
\end{align}

Using~\eqref{onesided mixed}, the partial expectation of the second term on the right side of~\eqref{11+12+13+14} can be expanded in the lower index~$\alpha$ to get
\begin{align} \label{121+122+123}
\E_{\alpha} \left[ \frac{G_{i \alpha} G_{\alpha s}}{G_{\alpha\alpha}} G^{(\alpha)}_{si} \right]& = \E_{\alpha} \left[ G_{\alpha\alpha} \sum_{k, l} G^{(\alpha)}_{ik} x_{\alpha k} x_{\alpha l} G^{(\alpha)}_{ls} G^{(\alpha)}_{si} \right] \nonumber\\
&= -\frac{1}{t_{\alpha}^{-1} -\tau} \frac{1}{N} \sum_k G_{ik}^{(\alpha)} G^{(\alpha)}_{ks} G^{(\alpha)}_{si}\nonumber\\ &\qquad + \frac{1}{(t_{\alpha}^{-1} -\tau)^2} \left(m^{(\alpha)}+\tau\right) \frac{1}{N} \sum_k G^{(\alpha)}_{ik} G^{(\alpha)}_{ks} G^{(\alpha)}_{si} \nonumber \\
&\qquad + \frac{2}{(t_{\alpha}^{-1} -\tau)^2} \frac{1}{N^2} \sum_{k, l} G^{(\alpha)}_{ik} G^{(\alpha)}_{kl} G^{(\alpha)}_{ls} G^{(\alpha)}_{si} + \caO(\Psi^5)\,.
\end{align}
Expanding the first term in the right side of~\eqref{121+122+123} further using~\eqref{basic resolvent}, we get
\begin{align} \label{1211+1212+1213+1214}
G^{(\alpha)}_{ik} G^{(\alpha)}_{ks} G^{(\alpha)}_{si} = G_{ik} G_{ks} G_{si} - \frac{G_{i \alpha} G_{\alpha k}}{G_{\alpha\alpha}} G_{ks} G_{si} - G^{(\alpha)}_{ik} \frac{G_{k \alpha} G_{\alpha s}}{G_{\alpha\alpha}} G_{si} - G^{(\alpha)}_{ik} G^{(\alpha)}_{ks} \frac{G_{s \alpha} G_{\alpha i}}{G_{\alpha\alpha}}\,.
\end{align}
We stop expanding the first term on the right side of~\eqref{1211+1212+1213+1214}, since it does no more contain the index~$a$, and we let
\begin{align} \label{X_33}
X_{33} \deq \frac{1}{N^2} \sum_{k, s} G_{ik} G_{ks} G_{si}\,.
\end{align}
Expanding the remaining terms on the right side of~\eqref{1211+1212+1213+1214} in the lower index~$\alpha$ using~\eqref{onesided mixed}, we obtain
\begin{align*}
\E_{\alpha} \left[ \frac{G_{i \alpha} G_{\alpha k}}{G_{\alpha\alpha}} G_{ks} G_{si} \right]& = -\frac{1}{t_{\alpha}^{-1} -\tau} \E_{\alpha} \left[ \sum_{l, m} G^{(\alpha)}_{il} x_{\alpha l} x_{\alpha m} G^{(\alpha)}_{mk} G_{ks} G_{si} \right] + \caO(\Psi^5)\nonumber \\
&= -\frac{1}{t_{\alpha}^{-1} -\tau} \frac{1}{N} \sum_l G^{(\alpha)}_{il} G^{(\alpha)}_{lk} G_{ks} G_{si} + \caO(\Psi^5)\nonumber \\&= -\frac{1}{t_{\alpha}^{-1} -\tau} \frac{1}{N} \sum_l G_{il} G_{lk} G_{ks} G_{si} + \caO(\Psi^5)
\end{align*}
and, similarly,
$$
\E_{\alpha} \left[ G^{(\alpha)}_{ik} \frac{G_{k \alpha} G_{\alpha s}}{G_{\alpha\alpha}} G_{si} \right] = -\frac{1}{t_{\alpha}^{-1} -\tau} \frac{1}{N} \sum_l G_{ik} G_{kl} G_{ls} G_{si} + \caO(\Psi^5)
$$
respectively,
$$
\E_{\alpha} \left[ G^{(\alpha)}_{ik} G^{(\alpha)}_{ks} \frac{G_{s \alpha} G_{\alpha i}}{G_{\alpha\alpha}} \right] = -\frac{1}{t_{\alpha}^{-1} -\tau} \frac{1}{N} \sum_l G_{ik} G_{ks} G_{sl} G_{li} + \caO(\Psi^5)\,.
$$
Thus, setting
\begin{align} \label{X_44}
X_{44} \deq \frac{1}{N^3} \sum_{k, l, s} G_{ik} G_{kl} G_{ls} G_{si}\,,
\end{align}
we have
\begin{align} \label{121}
\E_{\alpha} \left[ \frac{1}{N^2} \sum_{k, s} G_{ik}^{(\alpha)} G^{(\alpha)}_{ks} G^{(\alpha)}_{si} \right] = X_{33} + \frac{3}{t_{\alpha}^{-1} -\tau} X_{44} + \caO(\Psi^5)\,.
\end{align}

Next, we consider the $\caO(\Psi^4)$ terms on the right side of~\eqref{121+122+123}. Let
\begin{align} \label{X_43}
X_{43} \deq \left(m+\tau\right) \frac{1}{N^2} \sum_{k, s} G_{ik} G_{ks} G_{si}\,.
\end{align}
Then, we have for the second term on the right side of~\eqref{121+122+123} that
\begin{align} \label{122}
\left(m^{(\alpha)}+\tau\right) \frac{1}{N^2} \sum_{k, s} G^{(\alpha)}_{ik} G^{(\alpha)}_{ks} G^{(\alpha)}_{si} = X_{43} + \caO(\Psi^5)\,.
\end{align}
The last term on the right side of~\eqref{121+122+123} is simply estimated by
\begin{align} \label{123}
\frac{1}{N^3} \sum_{k, l, s} G^{(\alpha)}_{ik} G^{(\alpha)}_{kl} G^{(\alpha)}_{ls} G^{(\alpha)}_{si} = X_{44} + \caO(\Psi^5)\,.
\end{align}
In sum, we find 
\begin{align} \label{12}
\E_{\alpha} \left[ \frac{1}{N} \sum_{s} \frac{G_{i \alpha} G_{\alpha s}}{G_{\alpha\alpha}} G^{(\alpha)}_{si} \right] = -\frac{1}{t_{\alpha}^{-1} -\tau} X_{33} + \frac{1}{(t_{\alpha}^{-1} -\tau)^2} X_{43} - \frac{1}{(t_{\alpha}^{-1} -\tau)^2} X_{44} + \caO(\Psi^5)\,.
\end{align}
Similarly, we also have
\begin{align} \label{13}
\E_{\alpha} \left[ \frac{1}{N} \sum_{s} G^{(\alpha)}_{is} \frac{G_{s \alpha} G_{\alpha i}}{G_{\alpha\alpha}} \right] = -\frac{1}{t_{\alpha}^{-1} -\tau} X_{33}+ \frac{1}{(t_{\alpha}^{-1} -\tau)^2} X_{43} - \frac{1}{(t_{\alpha}^{-1} -\tau)^2} X_{44} + \caO(\Psi^5)\,.
\end{align}

For the last term in~\eqref{11+12+13+14}, we obtain
\begin{align*}
\frac{G_{i \alpha} G_{\alpha s}}{G_{\alpha\alpha}} \frac{G_{s \alpha} G_{\alpha i}}{G_{\alpha\alpha}} = G_{\alpha\alpha}^2 \sum_{k, l, p, q} G^{(\alpha)}_{ik} x_{\alpha k} x_{\alpha l} G^{(\alpha)}_{ls} G^{(\alpha)}_{sp} x_{\alpha p} x_{\alpha q} G^{(\alpha)}_{qi}\,.
\end{align*}
Hence, denoting
\begin{align} \label{X_44'}
X_{44}' \deq \frac{1}{N^3} \sum_{k, l, s} G_{is} G_{si} G_{kl} G_{lk}\,,
\end{align}
we find
\begin{align} \label{14}
\E_{\alpha} \left[ \frac{1}{N} \sum_{s} \frac{G_{i \alpha} G_{\alpha s}}{G_{\alpha\alpha}} \frac{G_{s \alpha} G_{\alpha i}}{G_{\alpha\alpha}} \right] = \frac{2}{(t_{\alpha}^{-1} -\tau)^2} X_{44} + \frac{1}{(t_{\alpha}^{-1} -\tau)^2} X_{44}' + \caO(\Psi^5)\,.
\end{align}

Thus, from~\eqref{11+12+13+14},~\eqref{12},~\eqref{13} and~\eqref{14} we obtain
\begin{multline} \label{1}
\E_{\alpha} \left[ \frac{1}{(t_{\alpha}^{-1} -\tau)^2} \frac{1}{N} \sum_{s} G_{is}^{(\alpha)} G_{si}^{(\alpha)} \right] = \frac{1}{(t_{\alpha}^{-1} -\tau)^2} X_{22} +\frac{2}{(t_{\alpha}^{-1} -\tau)^3} X_{33}\\ - \frac{2}{(t_{\alpha}^{-1} -\tau)^4} X_{43} - \frac{1}{(t_{\alpha}^{-1} -\tau)^4} X_{44}' + \caO(\Psi^5)\,,
\end{multline}
which completes the expansion of the first term in~\eqref{1+2+3+4+5+6+7}. The calculation and the result coincide with those in the deformed Wigner case in~\cite{LSY}, except the sign of the $X_{33}$ term. The discrepancy is due to the sign difference in the coefficient $(t_{\alpha}^{-1} -\tau)^{-1}$.

Adapting the expansion procedure of~\cite{LSY} we conclude, with the definitions
\begin{align} \label{X_32}
X_{32} \deq \left(m+\tau\right) \frac{1}{N} \sum_{s} G_{is} G_{si}\,,\qquad\quad X_{42} \deq\left( m+\tau\right)^2 \frac{1}{N} \sum_{s} G_{is} G_{si}\,,
\end{align}
that
\begin{multline}\label{expansion 1}
\E_{\alpha} [G_{i\alpha}G_{\alpha i}] = \frac{1}{(t_{\alpha}^{-1} -\tau)^2} X_{22} - \frac{2}{(t_{\alpha}^{-1} -\tau)^3} X_{32} - \frac{2}{(t_{\alpha}^{-1} -\tau)^3} X_{33} + \frac{3}{(t_{\alpha}^{-1} -\tau)^4} X_{42} \\+ \frac{6}{(t_{\alpha}^{-1} -\tau)^4} X_{43} + \frac{12}{(t_{\alpha}^{-1} -\tau)^4} X_{44}+ \frac{3}{(t_{\alpha}^{-1} -\tau)^4} X_{44}' + \caO(\Psi^5) \,.
\end{multline}
This shows~\eqref{le expansion in first lemma} and hence completes the proof of Lemma~\ref{le expansion in first lemma}.

\end{proof}

Before we move on to the next section, we introduce some more notation. For $k\in\N$, let
\begin{align}\label{le definition of Ak}
A_k \deq \frac{1}{N} \sum_{\rho} \frac{1}{(t_{\rho}^{-1} -\tau)^k}\,.
\end{align}
We remark that from~\eqref{tau 2}, we have
\begin{align}\label{le relation A2 and A3}
A_2 = \tau^{-2}\,,\qquad\qquad A_3+\tau^{-3}=1\,.
 \end{align}

Finally, averaging~\eqref{expansion 1} over~$\alpha$, we have in this notation
\begin{align} \label{expansion 1'}
\frac{1}{N} \sum_{\alpha} \E_{\alpha} \,[G_{i\alpha}G_{\alpha i}] = A_2 X_{22} - 2 A_3 ( X_{32} + X_{33} )+ 3 A_4 (X_{42} + 2 X_{43} + 4 X_{44} + X_{44}' ) + \caO(\Psi^5) \,.
\end{align}
This concludes the current appendix.

\section{Optical theorems}\label{sec:optical theorems}
In this section, we establish the following ``optical theorem''. 
\begin{lemma}\label{le lemma optical theorem}
 Under the assumptions of Proposition~\ref{prop green} with the notation of Lemma~\ref{le first lemma}, we have
\begin{align}\label{le eq lemma optical theorem}
2\E\, [X_{32}+X_{33}]-\frac{1}{N} &= 3(A_4 - \tau^{-4}) \E\,[X_{42}+2X_{43}+4X_{44}+X_{44}'] + \caO(\Psi^5)\,,
\end{align}
uniformly in $t\ge 0$.
\end{lemma}

Lemma~\ref{le lemma optical theorem} is an example of what we call optical theorems: optical theorems assure that the expectations of certain linear combinations of the random variables introduced in Lemma~\ref{le first lemma} are smaller than their naive sizes obtained from power counting using the local laws in Lemma~\ref{lem:square root for hat mfc} and Lemma~\ref{local law Greek}. Such estimates were key technical inputs in the proof of edge universality for deformed Wigner matrices in~\cite{LSY}. As in \cite{LSY}, the optical theorems used in this paper are obtained by combining expansions of random variables, {\it e.g.}, $X_{22}$ or $X_{33}$, with the sum rules in~\eqref{tau 2}. In the rest of this appendix, we derive the required optical theorems. 

The proof of Lemma~\ref{le lemma optical theorem} is given in the Subsection~\ref{last subsection optical theorems} based on estimates obtained in the Subsections~\ref{le subsection OT X22},~\ref{le subsection OT X32 X33} and~\ref{le subsection OT mX22}.

\subsection{Optical theorem from $X_{22}$}\label{le subsection OT X22}

To derive the first optical theorem, we consider
\begin{align}
\sum_{s}G_{is}G_{si}=G_{ii}^2+\sum_{s}^{(i)}G_{is}G_{si}\,.
\end{align}
Similar to the expansion of $G_{\alpha\alpha}$, we now expand $G_{ss}$ around $-\tau$. We notice that
$$
\bigg|\tau^{-1} -z -\sum_{\gamma, \delta} x_{\gamma s} G^{(s)}_{\gamma\delta} x_{\delta s} \bigg| \prec \Psi\,,
$$
which can be checked from~\eqref{L_+} and the estimate
$$
\bigg| G_{\alpha\alpha} - \frac{1}{-t_{\alpha}^{-1} +\tau} \bigg| \prec \Psi \,.
$$
Thus, using Schur's complement formula~\eqref{schur}, we obtain the following expansion of $G_{ss}$ in the lower index~$s$,
\begin{align*}
G_{ss} &= \frac{1}{h_{ss}-\sum_{\gamma, \delta}^{(s)} h_{\gamma s} G^{(s)}_{\gamma\delta} h_{s \delta}} = \frac{1}{-\tau^{-1} + \tau^{-1} -z -\sum_{\gamma, \delta} x_{\gamma s} G^{(s)}_{\gamma\delta} x_{\delta s}}\nonumber \\
&= -\tau - \tau^2 \left( \tau^{-1} -z -\sum_{\gamma, \delta} x_{\gamma s} G^{(s)}_{\gamma\delta} x_{\delta s} \right) - \tau^3 \left( \tau^{-1} -z -\sum_{\gamma, \delta} x_{\gamma s} G^{(s)}_{\gamma\delta} x_{\delta s} \right)^2 + \caO(\Psi^3)\,.
\end{align*}

Using the resolvent formula~\eqref{onesided} we therefore get the following expansion of $G_{is}G_{si}$  in the lower index~$s$, for~$s \neq i$,
\begin{align*}
&G_{is}G_{si}=G_{ss}^2 \sum_{\rho, \sigma} G^{(s)}_{i \rho} x_{\rho s} x_{\sigma s} G^{(s)}_{\sigma i} \nonumber \\
&\quad= \tau^2 \sum_{\rho, \sigma} G^{(s)}_{i \rho} x_{\rho s} x_{\sigma s} G^{(s)}_{\sigma i} + 2 \tau^3 \left( \tau^{-1} -z -\sum_{\gamma, \delta} x_{\gamma s} G^{(s)}_{\gamma\delta} x_{\delta s} \right) \sum_{\rho,\sigma} G^{(s)}_{i \rho} x_{\rho s} x_{\sigma s} G^{(s)}_{\sigma i} \nonumber\\
&\qquad\qquad + 3 \tau^4 \left( \tau^{-1} -z -\sum_{\gamma, \delta} x_{\gamma s} G^{(s)}_{\gamma\delta} x_{\delta s} \right)^2 \sum_{\rho,\sigma} G^{(s)}_{i \rho} x_{\rho s} x_{\sigma s} G^{(s)}_{\sigma i} + \caO(\Psi^5)\,.
\end{align*}
Taking the partial expectation $\E_s$, we obtain, for $s\not=i$,
\begin{align} \label{Ward1}
&\E_s [G_{is}G_{si}] = \frac{\tau^2}{N} \sum_{\rho} G^{(s)}_{i \rho} G^{(s)}_{\rho i} + \frac{2 \tau^3}{N} \left( \tau^{-1} -z - \frac{\widetilde m^{(s)}}{d} \right) \sum_{\rho} G^{(s)}_{i \rho} G^{(s)}_{\rho i} \nonumber \\ &\qquad- \frac{4 \tau^3}{N^2} \sum_{\rho, \sigma} G^{(s)}_{i \rho} G^{(s)}_{\rho \sigma} G^{(s)}_{\sigma i} + \frac{3 \tau^4}{N} \left( \tau^{-1} -z - \frac{\widetilde m^{(s)}}{d} \right)^2 \sum_{\rho} G^{(s)}_{i \rho} G^{(s)}_{\rho i}\nonumber \\
&\qquad- \frac{12 \tau^4}{N^2} \left( \tau^{-1} -z - \frac{\widetilde m^{(s)}}{d} \right) \sum_{\rho, \sigma} G^{(s)}_{i \rho} G^{(s)}_{\rho \sigma} G^{(s)}_{\sigma i} \nonumber \\
&\qquad + \frac{6 \tau^4}{N^3} \sum_{\rho, \sigma, \gamma} G^{(s)}_{i \rho} G^{(s)}_{\rho i} G^{(s)}_{\sigma\gamma} G^{(s)}_{\gamma\sigma} + \frac{24 \tau^4}{N^3} \sum_{\rho, \sigma, \gamma} G^{(s)}_{i \rho} G^{(s)}_{\rho \gamma} G^{(s)}_{\gamma\sigma} G^{(s)}_{\sigma i} + \caO(\Psi^5)\,.
\end{align}

Using the resolvent formula~\eqref{bad resolvent} to remove the upper indices~$s$ in~\eqref{Ward1}, we get, for $s \neq i$,
\begin{align} \label{expansion 2}
\E_s [G_{is} G_{si}] &= \frac{\tau^2}{N} \sum_{\rho} G_{i \rho} G_{\rho i} + \frac{2 \tau^3}{N} \left( \tau^{-1} -z - \frac{\widetilde m}{d} \right) \sum_{\rho} G_{i \rho} G_{\rho i}\nonumber\\ &\quad - \frac{2 \tau^3}{N^2} \sum_{\rho, \sigma} G_{i \rho} G_{\rho \sigma} G_{\sigma i} + \frac{3 \tau^4}{N} \left( \tau^{-1} -z - \frac{\widetilde m}{d} \right)^2 \sum_{\rho} G_{i \rho} G_{\rho i}\nonumber \\
&   \quad- \frac{6 \tau^4}{N^2} \left( \tau^{-1} -z - \frac{\widetilde m}{d} \right) \sum_{\rho, \sigma} G_{i \rho} G_{\rho \sigma} G_{\sigma i}\nonumber \\
& \quad+ \frac{12 \tau^4}{N^3} \sum_{\rho, \sigma, \gamma} G_{i \rho} G_{\rho \sigma} G_{\sigma \gamma} G_{\gamma i} + \frac{3 \tau^4}{N^3} \sum_{\rho, \sigma, \gamma} G_{i \rho} G_{\gamma \sigma} G_{\sigma \gamma} G_{\rho i} + \caO(\Psi^5)\,. 
\end{align}
We next expand all terms on the right side of~\eqref{expansion 2} except the first one to change Greek indices into Roman indices. Recall from~\eqref{le definition of Ak} that
\begin{align}
A_k = \frac{1}{N} \sum_{\rho} \frac{1}{(t_{\rho}^{-1} -\tau)^k}\,.
\end{align}

The last two terms on the right side of~\eqref{expansion 2} are easy to convert. For example,
\begin{align*}
G_{i \rho} G_{\rho \sigma} G_{\sigma \gamma} G_{\gamma i}& = G_{i \rho} G_{\rho \sigma} G^{(\rho)}_{\sigma \gamma} G^{(\rho)}_{\gamma i} + \caO(\Psi^5)\nonumber\\& = \frac{1}{(t_{\rho}^{-1} -\tau)^2} \sum_{j, k} G^{(\rho)}_{ij} x_{\rho j} x_{\rho k} G^{(\rho)}_{k \sigma} G^{(\rho)}_{\sigma \gamma} G^{(\rho)}_{\gamma i} + \caO(\Psi^5)\,,
\end{align*}
which shows that
\begin{align}
\E_{\rho} [ G_{i \rho} G_{\rho \sigma} G_{\sigma \gamma} G_{\gamma i}] = \frac{1}{(t_{\rho}^{-1} -\tau)^2} \frac{1}{N} \sum_j G_{ij} G_{j \sigma} G_{\sigma \gamma} G_{\gamma i} + \caO(\Psi^5)\,.
\end{align}
Repeating the argument once more we also find, using~\eqref{tau 2}, that
\begin{align} \label{w5}
\E \left[ \frac{12 \tau^4}{N^3} \sum_{\rho, \sigma, \gamma} G_{i \rho} G_{\rho \sigma} G_{\sigma \gamma} G_{\gamma i} \right]& = \frac{12 \tau^4}{N^3} A_2^3 \,\E \left[ \sum_{j, k, l} G_{ij} G_{jk} G_{kl} G_{li} \right] + \caO(\Psi^5) \nonumber \\
&= 12 \tau^{-2} \,\E\,[ X_{44}] + \caO(\Psi^5)\,.
\end{align}
Similarly,
\begin{align} \label{w6}
\E \left[ \frac{3 \tau^4}{N^3} \sum_{\rho, \sigma, \gamma} G_{i \rho}G_{\rho i} G_{\gamma \sigma} G_{\sigma \gamma}  \right] = 3 \tau^{-2}\, \E\, [X_{44}' ]+ \caO(\Psi^5)\,.
\end{align}
The other fourth order terms in~\eqref{expansion 2} require more treatment. We first consider
\begin{align*}
\tau^{-1} -z - \frac{\widetilde m}{d} &= \tau^{-1} -z - \frac{1}{N} \sum_{\beta} G_{\beta\beta} \nonumber \\
&= \tau^{-1} -z + \frac{1}{N} \sum_{\beta} \frac{1}{t_{\beta}^{-1} -\tau} - \frac{1}{N} \sum_{\beta} \frac{1}{(t_{\beta}^{-1} -\tau)^2} \left( \sum_{p, q} x_{\beta p} G^{(\beta)}_{pq} x_{\beta q}+\tau \right) + \caO(\Psi^2)\nonumber \\
&= (L_+ - z) - \frac{1}{N} \sum_{\beta} \frac{1}{(t_{\beta}^{-1} -\tau)^2} \left( \sum_{p, q} x_{\beta p} G^{(\beta)}_{pq} x_{\beta q}+\tau \right) + \caO(\Psi^2) \nonumber \\
&= -\frac{1}{N} \sum_{\beta} \frac{1}{(t_{\beta}^{-1} -\tau)^2} \left( \sum_{p, q} x_{\beta p} G^{(\beta)}_{pq} x_{\beta q}+\tau \right) + \caO(\Psi^2)\,. 
\end{align*}
We then obtain for the fifth term on the right side of~\eqref{expansion 2} that
\begin{align} \label{w3}
&\E \left[ -\frac{6 \tau^4}{N^2} \left( \tau^{-1} -z - \frac{\widetilde m}{d} \right) \sum_{\rho, \sigma} G_{i \rho} G_{\rho \sigma} G_{\sigma i} \right] \nonumber \\
&\qquad= \E \left[ \frac{6 \tau^4}{N^2} \left( \frac{1}{N} \sum_{\beta} \frac{1}{(t_{\beta}^{-1} -\tau)^2} \right) (m^{(\beta)} +\tau) \sum_{\rho, \sigma} G^{(\beta)}_{i \rho} G^{(\beta)}_{\rho \sigma} G^{(\beta)}_{\sigma i} \right] + \caO(\Psi^5) \nonumber\\
&\qquad =6\tau^{-2}\E\left[(m+\tau)\frac{1}{N^2}\sum_{k,l}G_{ik}G_{kl}G_{ki}\right]+ \caO(\Psi^5)\nonumber \\
&\qquad= 6 \tau^{-2} \E\,[ X_{43}] + \caO(\Psi^5)\,.
\end{align}
Similarly, we have for the forth therm on the right side of~\eqref{expansion 2} that
\begin{align} \label{w4}
\E \left[ \frac{3 \tau^4}{N} \left( \tau^{-1} -z - \frac{\widetilde m}{d} \right)^2 \sum_{\rho} G_{i \rho} G_{\rho i} \right] = 3 \tau^{-2} \E\, [X_{42}] + \caO(\Psi^5)\,.
\end{align}
This completes the discussion of the fourth order terms in~\eqref{expansion 2}.

We move on to the third order terms on the right side of~\eqref{expansion 2}. Adapting the expansion method above, we note that
\begin{align} \label{w11+w12+w13}
&\frac{1}{N} \left( \tau^{-1} -z - \frac{\widetilde m}{d} \right) \sum_{\rho} G_{i \rho} G_{\rho i} \nonumber \\
&\quad= -\frac{1}{N^2} \sum_{\beta} \frac{1}{(t_{\beta}^{-1} -\tau)^2} \left( \sum_{p, q} x_{\beta p} G^{(\beta)}_{pq} x_{\beta q}+\tau \right) \sum_{\rho} G_{i \rho} G_{\rho i} \nonumber\\
&\qquad + \frac{1}{N^2} \sum_{\beta} \frac{1}{(t_{\beta}^{-1} -\tau)^3} \left( \sum_{p, q} x_{\beta p} G^{(\beta)}_{pq} x_{\beta q}+\tau \right)^2 \sum_{\rho} G_{i \rho} G_{\rho i} + \frac{L_+ - z}{N} \sum_{\rho} G_{i \rho} G_{\rho i}+ \caO(\Psi^5)\,.
\end{align}
Taking the partial expectation $\E_{\beta}$, we get for the summand in the first term on the right side of~\eqref{w11+w12+w13} that
\begin{align} \label{w111+w112+w113}
&\E_{\beta} \left[ \frac{1}{(t_{\beta}^{-1} -\tau)^2} \left( \sum_{p, q} x_{\beta p} G^{(\beta)}_{pq} x_{\beta q}+\tau \right) G_{i \rho} G_{\rho i} \right] \nonumber\\
&\qquad= \frac{1}{(t_{\beta}^{-1} -\tau)^2} \left( m^{(\beta)}+\tau \right) G^{(\beta)}_{i \rho} G^{(\beta)}_{\rho i}\nonumber\\ &\qquad\qquad + \E_{\beta} \left[ \frac{1}{(t_{\beta}^{-1} -\tau)^2} \left( \sum_{p, q} x_{\beta p} G^{(\beta)}_{pq} x_{\beta q}+\tau \right) \frac{G_{i \beta} G_{\beta\rho}}{G_{\beta\beta}} G^{(\beta)}_{\rho i} \right] \nonumber \\
&\qquad\qquad + \E_{\beta} \left[ \frac{1}{(t_{\beta}^{-1} -\tau)^2} \left( \sum_{p, q} x_{\beta p} G^{(\beta)}_{pq} x_{\beta q}+\tau \right) G^{(\beta)}_{i \rho} \frac{G_{\rho\beta} G_{\beta i}}{G_{\beta\beta}} \right] + \caO(\Psi^5)\,. 
\end{align}
Expanding the first term on the right side of~\eqref{w111+w112+w113} with respect to the upper index $\beta$, we find
\begin{align} \label{w111}
&\E \left[ \frac{1}{(t_{\beta}^{-1} -\tau)^2} \left( m^{(\beta)}+\tau \right) G^{(\beta)}_{i \rho} G^{(\beta)}_{\rho i} \right] \nonumber \\
&\quad= \E \left[ \frac{1}{(t_{\beta}^{-1} -\tau)^2} \left( m+\tau \right) G_{i \rho} G_{\rho i} \right] + \E \left[ \frac{2}{(t_{\beta}^{-1} -\tau)^3} \left( m+\tau \right) \sum_k G_{ik} G_{k \rho} G_{\rho i} \right]\nonumber \\
&\quad\qquad + \E \left[ \frac{1}{(t_{\beta}^{-1} -\tau)^3} \sum_{k, l} G_{kl} G_{lk} G_{i \rho} G_{\rho i} \right] + \caO(\Psi^5)\,. 
\end{align}
Similarly, we get for the expectation of the second term on the right side of~\eqref{w111+w112+w113} that
\begin{align} \label{w112}
&\E \left[ \frac{1}{(t_{\beta}^{-1} -\tau)^2} \left( \sum_{p, q} x_{\beta p} G^{(\beta)}_{pq} x_{\beta q}+\tau \right) \frac{G_{i \beta} G_{\beta\rho}}{G_{\beta\beta}} G^{(\beta)}_{\rho i} \right]\nonumber \\
&\quad= -\E \left[ \frac{1}{(t_{\beta}^{-1} -\tau)^3} \left( m+\tau \right) \sum_k G_{ik} G_{k \rho} G_{\rho i} \right] - \E \left[ \frac{2}{(t_{\beta}^{-1} -\tau)^3} \frac{1}{N^2}\sum_{k, l} G_{ik} G_{kl} G_{l \rho} G_{\rho i} \right]\nonumber\\ &\qquad\qquad + \caO(\Psi^5)
\end{align}
and for the expectation of the third term on the right side of~\eqref{w111+w112+w113} that
\begin{align} \label{w113}
&\E \left[ \frac{1}{(t_{\beta}^{-1} -\tau)^2} \left( \sum_{p, q} x_{\beta p} G^{(\beta)}_{pq} x_{\beta q}+\tau \right) G^{(\beta)}_{i \rho} \frac{G_{\rho\beta} G_{\beta i}}{G_{\beta\beta}} \right] \nonumber\\
&\quad= -\E \left[ \frac{1}{(t_{\beta}^{-1} -\tau)^3} \left( m+\tau \right) \sum_k G_{ik} G_{k \rho} G_{\rho i} \right] - \E \left[ \frac{2}{(t_{\beta}^{-1} -\tau)^3} \frac{1}{N^2}\sum_{k, l} G_{ik} G_{kl} G_{l \rho} G_{\rho i} \right] \nonumber\\ &\qquad\qquad+ \caO(\Psi^5)\,.
\end{align}
Thus, from~\eqref{w111+w112+w113},~\eqref{w111},~\eqref{w112} and~\eqref{w113} we find
\begin{align} \label{w11}
&\E \left[ \frac{1}{(t_{\beta}^{-1} -\tau)^2} \left( \sum_{p, q} x_{\beta p} G^{(\beta)}_{pq} x_{\beta q}+\tau \right) G_{i \rho} G_{\rho i} \right] = \E \left[ \frac{1}{(t_{\beta}^{-1} -\tau)^2} \left( m+\tau \right) G_{s \rho} G_{\rho s} \right]\nonumber \\
&\qquad - \E \left[ \frac{4}{(t_{\beta}^{-1} -\tau)^3} \frac{1}{N^2}\sum_{k, l} G_{ik} G_{kl} G_{l \rho} G_{\rho i} \right] + \E \left[ \frac{1}{(t_{\beta}^{-1} -\tau)^3} \sum_{k, l} G_{kl} G_{lk} G_{i \rho} G_{\rho i} \right]	+ \caO(\Psi^5)\,.
\end{align}

We next remove the Greek index $\rho$ on the right side of~\eqref{w11}. We note that
\begin{align*}
\left( m+\tau \right) G_{i \rho} G_{\rho i} = G_{\rho\rho}^2 \left( m^{(\rho)} +\tau \right) \sum_{k, l} G^{(\rho)}_{ik} x_{\rho k} x_{\rho l} G^{(\rho)}_{li} + \frac{1}{N} \sum_j G_{j \rho} G_{\rho j} G_{\rho\rho} \sum_{k, l} G^{(\rho)}_{ik} x_{\rho k} x_{\rho l} G^{(\rho)}_{li}\,.
\end{align*}
Thus, taking the partial expectation $\E_{\rho}$, we find 
\begin{align}\label{le 8.21}
&\E_{\rho} \left[ \left( m+\tau \right) G_{i \rho} G_{\rho i} \right] = \frac{1}{(t_{\rho}^{-1} -\tau)^2} \left( m^{(\rho)} +\tau \right) \frac{1}{N} \sum_k G^{(\rho)}_{ik} G^{(\rho)}_{ki}\nonumber\\ &\qquad - \frac{2}{(t_{\rho}^{-1} -\tau)^3} \left( m^{(\rho)} +\tau \right)^2 \frac{1}{N} \sum_k G^{(\rho)}_{ik} G^{(\rho)}_{ki} \nonumber\\
&\qquad - \frac{4}{(t_{\rho}^{-1} -\tau)^3} \left( m^{(\rho)} +\tau \right) \frac{1}{N^2} \sum_{k, l} G^{(\rho)}_{ik} G^{(\rho)}_{kl} G^{(\rho)}_{li} \nonumber\\ &\qquad- \frac{2}{(t_{\rho}^{-1} -\tau)^3}\frac{1}{N^3} \sum_{j, k, l} G^{(\rho)}_{ik} G^{(\rho)}_{kj} G^{(\rho)}_{jl} G^{(\rho)}_{li} \nonumber \\
&\qquad - \frac{1}{(t_{\rho}^{-1} -\tau)^3}\frac{1}{N^3} \sum_{j, k, l} G^{(\rho)}_{ik} G^{(\rho)}_{ki} G^{(\rho)}_{jl} G^{(\rho)}_{lj} + \caO(\Psi^5)\,. 
\end{align}
Expanding the right side of~\eqref{le 8.21} with respect to the upper index $\rho$, we obtain 
\begin{multline} \label{w1111}
\E \left[ \left( m+\tau \right) G_{i \rho} G_{\rho i} \right]= \frac{1}{(t_{\rho}^{-1} -\tau)^2} \E\,[ X_{32}]\\- \frac{1}{(t_{\rho}^{-1} -\tau)^3} \big( 2 \E\, [X_{42} ]+ 2 \E\,[ X_{43} ]+ 2 \E\,[ X_{44} ]\big) + \caO(\Psi^5)\,. 
\end{multline}
We thus have for the first term on the right side of~\eqref{w11+w12+w13} that
\begin{align} \label{w11'}
&\E \left[ -\frac{1}{N^2} \sum_{\beta} \frac{1}{(t_{\beta}^{-1} -\tau)^2} \left( \sum_{p, q} x_{\beta p} G^{(\beta)}_{pq} x_{\beta q}+\tau \right) \sum_{\rho} G_{i \rho} G_{\rho i} \right] \nonumber \\
&\qquad= -\tau^{-4} \E\,[ X_{32} ]+ \tau^{-2} A_3 \big( 2 \E\,[ X_{42} ]+ 2 \E\,[ X_{43}] + 6 \E\,[ X_{44} ]- \E\,[ X_{44}'] \big)+ \caO(\Psi^5)\,.
\end{align}

The fourth order terms in~\eqref{w11+w12+w13} can easily be handled: we have
\begin{multline} \label{w12}
\E \left[ \frac{1}{N^2} \sum_{\beta} \frac{1}{(t_{\beta}^{-1} -\tau)^3} \left( \sum_{p, q} x_{\beta p} G^{(\beta)}_{pq} x_{\beta q}+\tau \right)^2 \sum_{\rho} G_{i \rho} G_{\rho i} \right]\\ = \tau^{-2} A_3 \big( \E \,[X_{42}] + 2 \E\, [X_{44}'] \big) + \caO(\Psi^5)\,,
\end{multline}
respectively,
\begin{align} \label{w13}
\E \left[ \frac{L_+ - z}{N} \sum_{\rho} G_{i \rho} G_{\rho i} \right] = \tau^{-2} (L_+ - z) \E\,[ X_{22}] + \caO(\Psi^5)\,.
\end{align}

We thus obtain from~\eqref{w11+w12+w13},~\eqref{w11'},~\eqref{w12} and~\eqref{w13} that
\begin{align} \label{w1}
\E \left[ \frac{2 \tau^3}{N} \left( \tau^{-1} -z - \frac{\widetilde m}{d} \right) \sum_{\rho} G_{i \rho} G_{\rho i} \right] &= -2 \tau^{-1} \E\,[ X_{32}] + 2 \tau A_3 \big( 3 \E\,[ X_{42}] + 2 \E\,[ X_{43}] + 6 \E\,[ X_{44}]\nonumber\\ &\qquad\qquad + \E\,[ X_{44}'] \big)+ 2\tau (L_+ - z) \E\,[ X_{22}] + \caO(\Psi^5)\,.
\end{align}

The third term on the right side of~\eqref{expansion 2}, which is also $\caO(\Psi^3)$, can be expanded in a similar manner: We begin with
\begin{align}\label{le 9.25}
G_{i \rho} G_{\rho \sigma} G_{\sigma i} = G_{i \rho} G_{\rho \sigma} G^{(\rho)}_{\sigma i} + G_{i \rho} G_{\rho \sigma} \frac{G_{i \rho} G_{\rho \sigma}}{G_{\rho\rho}}\,.
\end{align}
The second term on the right side of~\eqref{le 9.25} can easily be controlled: we have
\begin{align*}
\frac{1}{N^2} \sum_{\rho, \sigma} \E \left[ G_{i \rho} G_{\rho \sigma} \frac{G_{i \rho} G_{\rho \sigma}}{G_{\rho\rho}} \right] = -\tau^{-2} A_3 \big( 2 \E [X_{44}] + \E [X_{44}' ]\big) + \caO(\Psi^5)\,.
\end{align*}
Taking the partial expectation $\E_{\rho}$, we have
\begin{align}
&\E_{\rho} \left[ G_{i \rho} G_{\rho \sigma} G^{(\rho)}_{\sigma i} \right] \nonumber \\
&\qquad= \frac{1}{(t_{\rho}^{-1} -\tau)^2} \frac{1}{N} \sum_k G^{(\rho)}_{ik} G^{(\rho)}_{k \sigma} G^{(\rho)}_{\sigma i} - \frac{2}{(t_{\rho}^{-1} -\tau)^3} \frac{1}{N} \left( m^{(\rho)} + \tau \right) \sum_k G^{(\rho)}_{ik} G^{(\rho)}_{k \sigma} G^{(\rho)}_{\sigma i}\nonumber \\
& \qquad \qquad- \frac{4}{(t_{\rho}^{-1} -\tau)^3} \frac{1}{N^2} \sum_{k, l} G^{(\rho)}_{ik} G^{(\rho)}_{kl} G^{(\rho)}_{l \sigma} G^{(\rho)}_{\sigma i} + \caO(\Psi^5)\,. 
\end{align}
Thus, expanding with respect to the upper index $\rho$, we obtain
\begin{align}
\E_{\rho} \left[ G_{i \rho} G_{\rho \sigma} G^{(\rho)}_{\sigma s} \right]&= \frac{1}{(t_{\rho}^{-1} -\tau)^2} \frac{1}{N} \sum_{k} G_{ik} G_{k \sigma} G_{\sigma i} - \frac{2}{(t_{\rho}^{-1} -\tau)^3} \frac{1}{N} \left( m + \tau \right) \sum_{k} G_{ik} G_{k \sigma} G_{\sigma i}\nonumber \\
& \qquad\qquad - \frac{1}{(t_{\rho}^{-1} -\tau)^3} \frac{1}{N^2} \sum_{k, l} G_{ik} G_{kl} G_{l \sigma} G_{\sigma i} + \caO(\Psi^5)\,. 
\end{align}
Repeating the same procedure with $\sigma$ instead of $\rho$, we eventually find
\begin{multline} \label{w2}
\E \left[ \frac{2 \tau^3}{N^2} \sum_{\rho, \sigma} G_{s \rho} G_{\rho \sigma} G_{\sigma s} \right] = 2 \tau^{-1} \E\,[ X_{33}] - 2 \tau A_3 \big( 4\E\,[ X_{43}] + 6\E\, [X_{44} ]+ 2\E\,[ X_{44}'] \big) + \caO(\Psi^5)\,.
\end{multline}

We conclude from~\eqref{expansion 1'},~\eqref{expansion 2},~\eqref{w5},~\eqref{w6},~\eqref{w3},~\eqref{w4},~\eqref{w1} and~\eqref{w2} that
\begin{align*}
\frac{1}{N} \sum_{\alpha} \E\,[ G_{i \alpha} G_{\alpha i}] &= \frac{1}{N^2} \sum_s^{(i)} \sum_{\rho} \E\,[ G_{i \rho} G_{\rho i}] + \frac{\tau^{-2}}{N} \E\,[ G_{ii}^2] + 2 \tau^{-1} (L_+ - z) \E\,[ X_{22} ]\nonumber\\ &\qquad- 2 (A_3 + \tau^{-3}) \E\, [X_{32} + X_{33}]\nonumber \\
& \qquad+ 3(A_4 + \tau^{-4} + 2 \tau^{-1} A_3) \E\,[ X_{42} + 2 X_{43} + 4 X_{44} + X_{44}'] + \caO(\Psi^5)\,.
\end{align*}
Since
$$
\frac{1}{N^2} \sum_s^{(i)} \sum_{\rho} \E\,[ G_{i \rho} G_{\rho i}] = \frac{1}{N} \sum_{\rho} \E\,[ G_{i \rho} G_{\rho i}] + \caO(\Psi^5)\,,
$$
we obtain the relation
\begin{align} \label{X_22 ward'}
 2 (A_3 + \tau^{-3}) \E\, [X_{32} + X_{33}] &= 3(A_4 + \tau^{-4} + 2 \tau^{-1} A_3) \E\,[ X_{42} + 2 X_{43} + 4 X_{44} + X_{44}']\nonumber\\ &\qquad + \frac{\tau^{-2}}{N} \E\,[ G_{ii}^2] + 2 \tau^{-1} (L_+ - z) \E \,[X_{22}] + \caO(\Psi^5)\,. 
\end{align}

Recalling that
$$
G_{ii}^2 = \tau^2 + 2 \tau^3 \left( \tau^{-1} -z -\sum_{\gamma, \delta} x_{\gamma i} G^{(i)}_{\gamma\delta} x_{\delta i} \right) + \caO(\Psi^2)\,,
$$
we find 
\begin{align}\label{le Gii2}
\E \,[G_{ii}^2] = \tau^2 + 2\tau^3 \E \left[ \tau^{-1} -z - \frac{\wt m}{d} \right] + \caO(\Psi^2) = \tau^2 - 2\tau \E\,[m+\tau] + \caO(\Psi^2)\,.
\end{align}
Thus, plugging~\eqref{le Gii2} into~\eqref{X_22 ward'} and recalling from~\eqref{tau 2} that $A_3 + \tau^{-3}=1$, we find
\begin{align} \label{X_22 ward}
 2  \E\, [X_{32} + X_{33}] -\frac{  1}{N}
& = 3 (A_4 + \tau^{-4} + 2 \tau^{-1} A_3) \E\,[ X_{42} + 2 X_{43} + 4 X_{44} + X_{44}'] \nonumber\\ &\quad- \frac{2 \tau^{-1}}{N} \E\,[m+\tau] + 2 \tau^{-1} (L_+ - z) \E\,[ X_{22}] + \caO(\Psi^5)\,.
\end{align}	
The identity~\eqref{X_22 ward} is the optical theorem derived from $X_{22}$. We remark that the second and third term on the right  side of~\eqref{X_22 ward} are both $\caO(\Psi^4)$. In Subsection~\ref{le subsection OT mX22}, we show that they can be written as linear combinations of $X_{42}$, $X_{43}$, $X_{44}$ and $X_{44}'$.

\subsection{Optical theorems from $X_{32}$ and $X_{33}$}\label{le subsection OT X32 X33}
  
In a next step, we derive further optical theorems using the ideas presented in Subsection~\ref{le subsection OT X22}. We start by considering
\begin{align}\label{le X32 in OT}
X_{32} = (m+\tau) \frac{1}{N} \sum_{s} G_{is} G_{si} = (m+\tau) \frac{1}{N} \sum_{s}^{(i)} G_{is} G_{si} + (m+\tau) \frac{1}{N} G_{ii}^2 \,.
\end{align}
To estimate the first term on the very right side of~\eqref{le X32 in OT}, we consider, for $s \neq i$,
\begin{align} \label{32w}
(m+\tau) G_{is} G_{si} = (m^{(s)}+\tau) G_{is} G_{si} + \frac{1}{N} \sum_{j}^{(s)} \frac{G_{js} G_{sj}}{G_{ss}} G_{is} G_{si} + \caO(\Psi^5)\,.
\end{align}
We expand the first term on the right side of~\eqref{32w} with respect to the lower index $s$ to get
\begin{align}
(m^{(s)}+\tau) G_{is} G_{si} &= (m^{(s)}+\tau) G_{ss}^2 \sum_{\rho, \sigma} G^{(s)}_{i \rho} x_{\rho s} x_{\sigma s} G^{(s)}_{\sigma i} \nonumber \\
&= \tau^2 (m^{(s)}+\tau) \sum_{\rho, \sigma} G^{(s)}_{i \rho} x_{\rho s} x_{\sigma s} G^{(s)}_{\sigma i}\nonumber \\
&\qquad + 2 \tau^3 (m^{(s)}+\tau) \left( \tau^{-1} -z -\sum_{\gamma, \delta} x_{\gamma s} G^{(s)}_{\gamma\delta} x_{\delta s} \right) \sum_{\rho, \sigma} G^{(s)}_{i \rho} x_{\rho s} x_{\sigma s} G^{(s)}_{\sigma i} + \caO(\Psi^5)\,. \nonumber
\end{align}
Taking the partial expectation $\E_s$, we obtain
\begin{align*}
\E_s \left[ (m^{(s)}+\tau) G_{is} G_{si} \right] &= \frac{\tau^2}{N} (m^{(s)}+\tau) \sum_{\rho} G^{(s)}_{i \rho} G^{(s)}_{\rho i} + \frac{2 \tau^3}{N} (m^{(s)}+\tau) \left( \tau^{-1} -z - \frac{\wt m^{(s)}}{d} \right) \sum_{\rho} G^{(s)}_{i \rho} G^{(s)}_{\rho i} \nonumber\\
& \qquad\qquad- \frac{4 \tau^3}{N^2} (m^{(s)}+\tau) \sum_{\rho, \sigma} G^{(s)}_{i \rho} G^{(s)}_{\rho \sigma} G^{(s)}_{\sigma i} + \caO(\Psi^5)\,.
\end{align*}
Since
\begin{multline*}
\frac{\tau^2}{N} (m^{(s)}+\tau) \sum_{\rho} G^{(s)}_{i \rho} G^{(s)}_{\rho i} 
= \frac{\tau}{N} (m+\tau) \sum_{\rho} G_{i \rho} G_{\rho i}\\ + \frac{2 \tau}{N} (m+\tau) \sum_{\rho} G_{is} G_{s \rho} G_{\rho i} + \frac{\tau}{N^2} \sum_{j} \sum_{\rho} G_{js} G_{sj} G_{i \rho} G_{\rho i}+ \caO(\Psi^5) 
\end{multline*}
and since
$$
\E \left[ \frac{\tau^3}{{N^2}} (m+\tau) \sum_{\rho, \sigma} G_{i \rho} G_{\rho \sigma} G_{\sigma i} \right] = \E \left[ \frac{\tau}{{ N^2}} (m+\tau) \sum_{k} \sum_{\sigma} G_{ik} G_{k \sigma} G_{\sigma i} \right] + \caO(\Psi^5)\,,
$$
we obtain
\begin{multline*}
\E \left[ (m^{(s)}+\tau) \frac{1}{N} \sum_s^{(i)} G_{is} G_{si} \right] = \E \left[ \frac{\tau^2}{N} (m+\tau) \sum_{\rho} G_{i \rho} G_{\rho i} \right] - \tau^{-1} \E\,[ 2 X_{42} +2 X_{43} - X_{44}'] + \caO(\Psi^5)\,.
\end{multline*}
Moreover, we have that
$$
\E \left[ \frac{1}{N} \sum_{j}^{(s)} \frac{G_{js} G_{sj}}{G_{ss}} G_{is} G_{si} \right] = -\tau^{-1} \E\,[ 2 X_{44} + X_{44}'] + \caO(\Psi^5)\,.
$$
We thus find the relation
\begin{multline*}
\E\,[ X_{32}] - N^{-1}\E \left[ (m+\tau)  G_{ii}^2 \right] = \tau^2 \E \left[ (m+\tau) \frac{1}{N} \sum_{\rho} G_{i \rho} G_{\rho i} \right] - 2\tau^{-1} \E\,[ X_{42} + X_{43} + X_{44}] + \caO(\Psi^5)\,.
\end{multline*}
Applying~\eqref{w1111}, we obtain
\begin{multline}\label{le 9.39}
\E \,[X_{32}] -N^{-1} \E \left[ (m+\tau)  G_{ii}^2 \right] = \E\,[ X_{32}] - 2 (\tau^2 A_3 + \tau^{-1}) \E\, [X_{42} + X_{43} + X_{44}] + \caO(\Psi^5)\,.
\end{multline}
Further, since
$$
N^{-1}\E \left[ (m+\tau) G_{ii}^2 \right] = \tau^2 N^{-1} \E \,[ m+\tau ] + \caO(\Psi^5)\,,
$$
we obtain from~\eqref{le 9.39} the identity
\begin{align} \label{X_32 ward}
N^{-1} \E\, [ m+\tau ] = 2 (A_3 + \tau^{-3}) \E\, [X_{42} + X_{43} + X_{44}] + \caO(\Psi^5)\,,
\end{align}
which is the optical theorem derived from $X_{32}$.

We next derive the optical theorem obtained from 
\begin{align}\label{le X33 OT}
X_{33} = \frac{1}{N^2} \sum_{k, s} G_{ik} G_{ks} G_{si}\,.
\end{align}
Since the contributions to the sums in~\eqref{le X33 OT} from the cases $i=k$ or $s=k$ are negligible (of $\caO(\Psi^2)$), we assume that $i, s \neq k$. Expanding the summand in~\eqref{le X33 OT} with respect to the lower index $k$, we get
\begin{align}\label{le 9.41}
G_{ik} G_{ks} G_{si} = G_{kk}^2 \sum_{\rho, \sigma} G^{(k)}_{i \rho} x_{\rho k} x_{\sigma k} G^{(k)}_{\sigma s} G^{(k)}_{si} + G_{ik} G_{ks} \frac{G_{sk} G_{ki}}{G_{kk}}\,.
\end{align}
Taking the partial expectation~$\E_k$, we find for the first term on the right side of~\eqref{le 9.41} that
\begin{align*}
&\E \left[ G_{kk}^2 \sum_{\rho, \sigma} G^{(k)}_{i \rho} x_{\rho k} x_{\sigma k} G^{(k)}_{\sigma s} G^{(k)}_{si} \right] \nonumber \\
&\quad= \E \left[ \frac{\tau^2}{N} \sum_{\rho} G^{(k)}_{i \rho} G^{(k)}_{\rho s} G^{(k)}_{si} \right] + \E \left[ \frac{2 \tau^3}{N} \left( \tau^{-1} -z - \frac{\wt m^{(k)}}{d} \right) \sum_{\rho} G^{(k)}_{i \rho} G^{(k)}_{\rho s} G^{(k)}_{si} \right]\nonumber \\
&\qquad - \E \left[ \frac{4 \tau^3}{N^2} \sum_{\rho, \sigma} G^{(k)}_{i \rho} G^{(k)}_{\rho \sigma} G^{(k)}_{\sigma s} G^{(k)}_{si} \right] + \caO(\Psi^2)\,. 
\end{align*}
Expanding further with respect to the upper index $k$, we thus find from~\eqref{le 9.41} that
\begin{align} \label{X_33 with gamma 1}
\E\,[ X_{33} ]= \E \left[ \frac{\tau^2}{N^2} \sum_{s} \sum_{\rho} G_{i \rho} G_{\rho s} G_{si} \right] - \tau^{-1} \E\, [2 X_{43} + 3 X_{44} + X_{44}'] + \caO(\Psi^5)\,.
\end{align}
Expanding the summand in the first term on the right of~\eqref{X_33 with gamma 1} with respect the index $\rho$, we get
\begin{align} \label{X_33 with gamma 2}
\E\,[ X_{33}] = \E\, [X_{33} ]- (\tau^2 A_3 + \tau^{-1}) \E\, [2 X_{43} + 3 X_{44} + X_{44}'] + \caO(\Psi^5)\,,
\end{align}
that is, recalling $A_3+\tau^{-3}=1$ (see~\eqref{le relation A2 and A3}),
\begin{align} \label{X_33 ward}
\tau^{-2}\E\, [2 X_{43} + 3 X_{44} + X_{44}'] = \caO(\Psi^5)\,,
\end{align}
which is the optical theorem derived from~$X_{33}$. 

\subsection{Optical theorem from $mX_{22}$}\label{le subsection OT mX22}
We return to the concluding remarks of Subsection~\ref{le subsection OT X22}. In the present subsection, we show that the terms $(L_+ - z) \E \,[X_{22}]$ and $N^{-1} \E\,[m+\tau]$, both appearing in~\eqref{X_22 ward}, can be decomposed into linear combinations of $X_{42}$, $X_{43}$, $X_{44}$ and $X_{44}'$. The latter term, $N^{-1} \E\,[m+\tau]$, can be handled by~\eqref{X_32 ward}, while the former needs to be dealt with the optical theorem obtained from~$mX_{22}$. Recall that
\begin{align}\label{le mX22}
m X_{22} &= \frac{1}{N^2} \sum_{a, s} G_{aa} G_{is} G_{si} \,.
\end{align}
Expanding the summand on the right side of~\eqref{le mX22} in the index $a$, we get
\begin{align}\label{le 9.46}
mX_{22}&= \frac{1}{N^2} \sum_{a \neq s} \left( G_{aa} G^{(a)}_{is} G^{(a)}_{si} + G^{(a)}_{is} G_{sa} G_{ai} + G_{ia} G_{as} G_{si} \right) + \caO(\Psi^5) \nonumber\\
&= \frac{1}{N^2} \sum_{a \neq s} \left( G_{aa} G^{(a)}_{is} G^{(a)}_{si} - \frac{G_{ia} G_{as}}{G_{aa}} G_{sa} G_{ai} \right) + 2X_{33} + \caO(\Psi^5)\,.
\end{align}
Expanding the second summand of the first term on the right side of~\eqref{le 9.46} with respect the lower index~$a$, we get
\begin{multline} \label{mX}
\E\, [m X_{22}] = \E \left[ \frac{1}{N^2} \sum_{a \neq s} G_{aa} G^{(a)}_{is} G^{(a)}_{si} \right] + \tau^{-1} \E\,[ 2X_{44} + X_{44}'] + 2\E\,[ X_{33}] + \caO(\Psi^5)\,.
\end{multline}
We expand the summand of the first term on the right of~\eqref{mX} further in the lower index $a$ to find
\begin{align} \label{m1+m2+m3+m4}
\E_a [G_{aa} G^{(a)}_{is} G^{(a)}_{si}] &= -\tau G^{(a)}_{is} G^{(a)}_{si} - \tau^2 \left( \tau^{-1} - z - \frac{\wt m^{(a)}}{d} \right) G^{(a)}_{is} G^{(a)}_{si} - \tau^3 \left( \tau^{-1} - z - \frac{\wt m^{(a)}}{d} \right)^2 G^{(a)}_{is} G^{(s)}_{si}\nonumber\\
&\qquad\qquad- \frac{2 \tau^3}{N^2} \sum_{\gamma, \delta} G^{(a)}_{\gamma\delta} G^{(a)}_{\delta\gamma} G^{(a)}_{is} G^{(a)}_{si} + \caO(\Psi^5)\,.
\end{align}
Expanding the first term on the right side of~\eqref{m1+m2+m3+m4} with respect the upper index $a$, we get
\begin{align} \label{m11+m12+m13+m14}
G^{(a)}_{is} G^{(a)}_{si} = G_{is} G_{si} - G^{(a)}_{is} \frac{G_{sa} G_{ai}}{G_{aa}} - \frac{G_{ia} G_{as}}{G_{aa}} G^{(a)}_{si} - \frac{G_{ia} G_{as}}{G_{aa}} \frac{G_{sa} G_{ai}}{G_{aa}}.
\end{align}
We stop expanding the first term on the right side of~\eqref{m11+m12+m13+m14} which will eventually, after averaging over~$s$,  become~$X_{22}$. For the second term on the right side of~\eqref{m11+m12+m13+m14}, we have
\begin{align*}
\E_a \left[ G^{(a)}_{is} \frac{G_{sa} G_{ai}}{G_{aa}} \right] &= -\frac{\tau}{N} \sum_{\gamma} G^{(a)}_{is} G^{(a)}_{s \gamma} G^{(a)}_{\gamma i} - \frac{\tau^2}{N} \left( \tau^{-1} - z - \frac{\wt m^{(a)}}{d} \right) \sum_{\gamma} G^{(a)}_{is} G^{(a)}_{s \gamma} G^{(a)}_{\gamma i}\nonumber\\ &\qquad\qquad+ \frac{2\tau^2}{N^2} \sum_{\gamma, \delta} G^{(a)}_{is} G^{(a)}_{s \gamma} G^{(a)}_{\gamma \delta} G^{(a)}_{\delta i} + \caO(\Psi^5)\,.
\end{align*}
Thus,
\begin{align}\label{le 9.51}
\E \left[ \frac{\tau}{N^2} \sum_{a \neq s} G^{(a)}_{is} \frac{G_{sa} G_{ai}}{G_{aa}} \right] &= -\E \left[ \frac{\tau^2}{N^3} \sum_{a \neq s} \sum_{\gamma} G^{(a)}_{is} G^{(a)}_{s \gamma} G^{(a)}_{\gamma i} \right] + \tau^{-1} \E\, [ X_{43} + 2 X_{44}] + \caO(\Psi^5) \nonumber \\
&= -\E \left[ \frac{\tau^2}{N^2} \sum_{s} \sum_{\gamma} G_{is} G_{s \gamma} G_{\gamma i} \right] + \tau^{-1} \E\, [ X_{43} - X_{44}] + \caO(\Psi^5)\,.
\end{align}
Following the calculation in~\eqref{X_33 with gamma 1}--\eqref{X_33 with gamma 2}, we obtain from~\eqref{le 9.51} that
\begin{align}
\E \left[ \frac{\tau}{N^2} \sum_{a \neq s} G^{(a)}_{is} \frac{G_{sa} G_{ai}}{G_{aa}} \right] &= -\E\,[ X_{33} ]+ \tau^2 A_3 \E \,[ 2X_{43} + 3X_{44} + X_{44}']\nonumber\\
&\qquad\qquad + \tau^{-1} \E \,[ X_{43} - X_{44}] + \caO(\Psi^5)\,.
\end{align}
The third term on the right side of~\eqref{m11+m12+m13+m14} can be expanded in a similar manner. In sum, we get
\begin{align} \label{m1}
-\E \left[ \frac{\tau}{N^2} \sum_{a \neq s} G^{(a)}_{is} G^{(a)}_{si} \right] &= -\tau \E\,[ X_{22}] -2 \E\,[ X_{33}] + 2 \tau^2 A_3 \E \,[ 2X_{43} + 3X_{44} + X_{44}']\nonumber \\
&\qquad\qquad + \tau^{-1} \E\, [ 2 X_{43} + X_{44}'] + \caO(\Psi^5)\,.
\end{align}

We next consider the second term on the right side of~\eqref{m1+m2+m3+m4}. We note that
\begin{align} \label{m21+m22+m23+m24}
&\left( \tau^{-1} - z - \frac{\wt m^{(a)}}{d} \right) G^{(a)}_{is} G^{(a)}_{si}= \left( \tau^{-1} - z - \frac{\wt m}{d} \right) G_{is} G_{si} \nonumber\\ &\qquad\qquad+ \tau^{-1} \left( \tau^{-1} - z - \frac{\wt m}{d} \right) G_{ia} G_{as} G_{si}\nonumber \\
&\qquad \qquad+ \tau^{-1} \left( \tau^{-1} - z - \frac{\wt m}{d} \right) G_{is} G_{sa} G_{ai} - \frac{\tau^{-1}}{N} \sum_{\gamma} G_{\gamma a} G_{a \gamma} G_{is} G_{si} + \caO(\Psi^5)\,. 
\end{align}
We expand the first term on the right side of~\eqref{m21+m22+m23+m24} similar to~\eqref{w11+w12+w13} to get
\begin{align} \label{w111+w112+113}
&\left( \tau^{-1} -z - \frac{\widetilde m}{d} \right) G_{is} G_{si} = -\frac{1}{N} \sum_{\beta} \frac{1}{(t_{\beta}^{-1} -\tau)^2} \left( \sum_{p, q} x_{\beta p} G^{(\beta)}_{pq} x_{\beta q}+\tau \right) G_{is} G_{si} \nonumber\\
&\quad+ \frac{1}{N} \sum_{\beta} \frac{1}{(t_{\beta}^{-1} -\tau)^3} \left( \sum_{p, q} x_{\beta p} G^{(\beta)}_{pq} x_{\beta q}+\tau \right)^2 G_{is} G_{si} + (L_+ - z) G_{is} G_{si} + \caO(\Psi^5)\,.
\end{align}
Taking the partial expectation $\E_{\beta}$ and proceeding as in~\eqref{w111+w112+w113}--\eqref{w113} we find for the first term on the right side of~\eqref{w111+w112+113} that
\begin{align*}
\E \left[ \frac{\tau^2}{N^3} \sum_{i, s} \sum_{\beta} \frac{1}{(t_{\beta}^{-1} -\tau)^2} \left( \sum_{p, q} x_{\beta p} G^{(\beta)}_{pq} x_{\beta q}+\tau \right) G_{is} G_{si} \right] = \E\,[ X_{32}]+ \tau^2 A_3 \E\, [X_{44}' -4 X_{44}] + \caO(\Psi^5)\,.
\end{align*}
We thus have
\begin{multline} \label{m2}
\E \left[ -\frac{\tau^2}{N^2} \sum_{a \neq s} \left( \tau^{-1} - z - \frac{\wt m^{(a)}}{d} \right) G^{(a)}_{is} G^{(a)}_{si} \right]=
 \E\,[ X_{32}] \\- \tau^2 A_3 \E\, [X_{42} +4 X_{44} + X_{44}']  - \tau^2 (L_+ -z) \E\,[ X_{22}] + \tau^{-1} \E \,[2X_{43} + X_{44}'] + \caO(\Psi^5)\,. 
\end{multline}

From~\eqref{m1+m2+m3+m4},~\eqref{m1} and~\eqref{m2} we find for the first term on the right side of~\eqref{mX} that
\begin{align}\label{le 9.55}
&\E \left[ \frac{1}{N^2} \sum_{a \neq s} G_{aa} G^{(a)}_{is} G^{(a)}_{si} \right] \nonumber \\
&\quad= -\tau \E \,[X_{22}] -2 \E\,[ X_{33}] + 2 \tau^2 A_3 \E\, [ 2X_{43} + 3X_{44} + X_{44}'] + \tau^{-1} \E \,[ 2 X_{43} + X_{44}']\nonumber \\
& \qquad+ \E\,[ X_{32} ]- \tau^2 A_3 \E\, [X_{42} +4 X_{44} + X_{44}'] - \tau^2 (L_+ -z) \E\,[ X_{22}] + \tau^{-1} \E \,[2X_{43} + X_{44}'] \nonumber \\
&\qquad -\tau^{-1} \E \,[X_{42} + 2 X_{44}'] + \tau^{-1} \E\,[ 2X_{44} + X_{44}'] + 2\E\,[ X_{33} ]+ \caO(\Psi^5)\,. 
\end{align}
Plugging~\eqref{le 9.55} into~\eqref{mX}, we finally find
\begin{multline*}
\E\,[m X_{22}] + \tau^2 (L_+ -z) \E\,[ X_{22} ]= -\tau \E \,[X_{22} ]+ \E\,[ X_{32}] \\ + (\tau^2 A_3 + \tau^{-1}) \E\, [ -X_{42} + 4X_{43} + 2X_{44} + X_{44}']  + \caO(\Psi^5)\,. 
\end{multline*}
Since $X_{32} = (m+\tau) X_{22}$ by definition, we obtain
\begin{align} \label{mX_22 ward}
(L_+ -z) \E\,[ X_{22}] = (A_3 + \tau^{-3}) \E\, [ -X_{42} + 4X_{43} + 2X_{44} + X_{44}']  + \caO(\Psi^5)\,,
\end{align}
which is the optical theorem obtained from $mX_{22}$.

\subsection{Proof of Lemma~\ref{le lemma optical theorem}}\label{last subsection optical theorems}
In this subsection, we prove Lemma~\ref{le lemma optical theorem} based on the optical theorems derived in the Subsections~\ref{le subsection OT X22},~\ref{le subsection OT X32 X33} and~\ref{le subsection OT mX22}.
\begin{proof}[Proof of Lemma~\ref{le lemma optical theorem}]
 For simplicity set
\begin{align}\label{le definition X3 and X4}
 X_3\deq 2(X_{32}+X_{33})\,,\qquad X_4\deq 3(X_{42}+2X_{43}+4X_{44}+X_{44}')\,.
\end{align}

From~\eqref{X_22 ward},~\eqref{X_32 ward} and~\eqref{mX_22 ward} we have
\begin{multline*}
\E \,[X_3] - N^{-1}= (A_4 + \tau^{-4} + 2 \tau^{-1} A_3) \E\,[X_4] - 2\tau^{-1} N^{-1} \E\,[m+\tau] \\+ 2 \tau^{-1} (L_+ - z)  \E\,[ X_{22}] + \caO(\Psi^5) \,,
\end{multline*}
hence,
\begin{align}\label{le to be subtracted eight}
\E \,[X_3] - N^{-1}= (A_4 + \tau^{-4} + 2 \tau^{-1} A_3) \E\,[X_4]  - \tau^{-1} \E\,[ 6 X_{42} - 4 X_{43} - 2 X_{44}'] + \caO(\Psi^5)\,.
\end{align}
Subtracting $8$-times~\eqref{X_33 ward} from~\eqref{le to be subtracted eight}, we obtain
\begin{align*}
\E \,[X_3]- N^{-1} &= (A_4 + \tau^{-4} + 2 \tau^{-1} A_3) \E\,[X_4]- \tau^{-1}  6\E\,[  X_{42} + 2 X_{43} + 4 X_{44} +  X_{44}'] + \caO(\Psi^5)\nonumber \\
&= (A_4 + \tau^{-4} + 2 \tau^{-1} A_3 -2 \tau^{-1}) \E\,[X_4] + \caO(\Psi^5)\,.
\end{align*}
Using $A_3 + \tau^{-3} = 1$ (see~\eqref{le relation A2 and A3}), we conclude that
\begin{align} \label{ward final}
\E\, [X_3]- N^{-1} = (A_4 - \tau^{-4}) \E\,[X_4] + \caO(\Psi^5)\,.
\end{align}
This proves~\eqref{le eq lemma optical theorem} and concludes the proof of Lemma~\ref{le lemma optical theorem}.

\end{proof}

\section{Proof of Lemma 6.1}\label{sec:proof of getting rid of the rest}

In this last section, we prove Lemma~\ref{getting rid of the rest}.

\begin{proof}[Proof of Lemma~\ref{getting rid of the rest}]

In a first step of the proof of~\eqref{le eq getting rid of the rest}, we express $(\partial_t t_{\alpha})/t_\alpha^2$ in terms of  $\gamma$ and $\dot\gamma$.

From the time evolution of $\Sigma=\diag(\sigma_\alpha)$ in~\eqref{le time evolution of Sigma}, we have
$$
\partial_t \frac{1}{\sigma_{\alpha}(t)} = -\e{-t} \frac{1}{\sigma_{\alpha}(0)} + \e{-t} = 1 - \frac{1}{\sigma_{\alpha}(t)}\,,\qquad\qquad (t\ge 0)\,.
$$
Since $t_{\alpha} = \gamma \sigma_{\alpha}$ by the definition of $T$, we get
$$
\frac{\partial_t t_{\alpha}}{t_{\alpha}^2} = - \partial_t\frac{1}{ t_{\alpha}(t)}=\left( \frac{\dot\gamma}{\gamma} + 1 \right)\frac{1}{ t_{\alpha}(t)} - \frac{1}{\gamma}\,,\qquad\qquad(t\ge0)\,.
$$
Recalling the definitions of $(A_k)$ in~\eqref{le definition of Ak} and that $A_2=\tau^{-2}$, we then obtain, dropping for simplicity the $t$-dependence from the notation,
\begin{align}\label{le relation among coefs 1}
\frac{1}{N}\sum_{\alpha} \frac{\partial_t t_{\alpha}}{t_{\alpha}^2} \frac{1}{(t_{\alpha}^{-1} -\tau)^3} = \left( \frac{\dot\gamma}{\gamma} + 1 \right)\tau^{-2} + \left( \frac{\dot\gamma}{\gamma} + 1 \right) \tau A_3- \frac{1}{\gamma} A_3\,,
\end{align}
respectively,
\begin{align}\label{le relation among coefs 2}
\frac{1}{N}\sum_{\alpha} \frac{\partial_t t_{\alpha}}{t_{\alpha}^3} \frac{1}{(t_{\alpha}^{-1} -\tau)^4} = \left( \frac{\dot\gamma}{\gamma} + 1 \right) A_3 + \left( \frac{\dot\gamma}{\gamma} + 1 \right) \tau A_4 - \frac{1}{\gamma}A_4\,.
\end{align}
Using the short-hand notation
\begin{align}
 X_3 =2(X_{32} + X_{33})\,,\qquad X_4 =3 (X_{42} + 2X_{43} +4X_{44} + X_{44}')\,,
\end{align}
 (see~\eqref{le definition X3 and X4}), we observe that~\eqref{le eq getting rid of the rest} is proven, once we have established that
\begin{multline}\label{claim}
\left[ \left( {\dot\gamma} + \gamma \right) \tau^{-2} + \left( {\dot\gamma \tau} + \gamma\tau - 1 \right) A_3 \right] \im \E\,[X_3]\\ =\left[ \left({\dot\gamma}+\gamma \right) A_3 + \left( {\dot\gamma \tau} + \gamma\tau - 1\right) A_4 \right]\,\im\E\,[X_4]+\caO(\Psi^5)\,.
\end{multline}
Combining the following lemma with Lemma~\ref{le lemma optical theorem}, it is straightforward to assure the validity~\eqref{claim}.
\begin{lemma}\label{le lemma coeffs}
Let $\gamma$ and $\tau$ be defined in~\eqref{le gamma} and~\eqref{le tau}. Then we have
\begin{align}
 &\left( {\dot\gamma} +\gamma \right) \tau^{-2} + \left( {\dot\gamma \tau}+\gamma \tau -1 \right) A_3 =\gamma(\tau^{-2} A_4 - A_3^2)\,,\label{le lemma coeffs 1}\\
& \left({\dot\gamma}+ \gamma \right) A_3 + \left( {\dot\gamma \tau} + \gamma\tau -1 \right) A_4=\gamma (\tau^{-2} A_4 - A_3^2)(A_4 - \tau^{-4}) \label{le lemma coeffs 2} \,.
\end{align}

\end{lemma}
Assuming the correctness of Lemma~\ref{le lemma coeffs}, we can recast~\eqref{claim} as
\begin{align}\label{le to be verified}
\gamma(\tau^{-2} A_4 - A_3^2)\im \E\,[X_3]&=\gamma(\tau^{-2} A_4 - A_3^2)(A_4 - \tau^{-4})\,\im\E\,[X_4]+\caO(\Psi^5)\,.
 \end{align}

Since $\E\,[X_3]-1/N=(A_4 - \tau^{-4})\E\,[X_4]+\caO(\Psi^5)$, by the optical theorem~\eqref{le eq lemma optical theorem}, we see that~\eqref{le to be verified}, respectively~\eqref{claim}, indeed hold true. This in turn proves, by the discussion above, the claim in~\eqref{le eq getting rid of the rest}, {\it i.e.,} Lemma~\ref{getting rid of the rest}.

It remains to prove Lemma~\ref{le lemma coeffs}:
\begin{proof}[Proof of Lemma~\ref{le lemma coeffs}]
First, we differentiate the sum rule
$$
\frac{1}{N} \sum_{\alpha} \left( \frac{1}{t_{\alpha}^{-1} - \tau} \right)^2 = \frac{1}{\tau^2}\,,
$$
(see~\eqref{tau 2}) with respect to $t$ to find
$$
\frac{\dot\tau}{\tau^3} = \frac{1}{N} \sum_{\alpha} \frac{\partial_t t_{\alpha}^{-1} - \dot\tau}{(t_{\alpha}^{-1} - \tau)^3} = -\dot\tau A_3 - \gamma^{-1} A_3 + \left( \frac{\dot\gamma}{\gamma} + 1 \right) \frac{1}{N} \sum_{\alpha} \frac{t_{\alpha}^{-1}}{(t_{\alpha}^{-1} - \tau)^3}\,,
$$
which yields
\begin{align}\label{le simplification of the coefs 1}
(A_3 + \tau^{-3}) \dot\tau = \gamma^{-1} [ (\dot\gamma + \gamma) (\tau^{-2} + \tau A_3) - A_3 ]\,.
\end{align}
Using $A_3+\tau^{-3}=1$, we hence get 
\begin{align}\label{le dot tau}
  \dot\tau = \gamma^{-1} [ (\dot\gamma + \gamma)\tau - A_3 ]\,.
\end{align}

Similarly, differentiating the sum rule
$$
\frac{1}{N} \sum_{\alpha} \left( \frac{1}{t_{\alpha}^{-1} - \tau} \right)^3 + \frac{1}{\tau^3} = 1\,,
$$
(see~\eqref{tau 2}) with respect to $t$ we find
\begin{align*}
(A_4 - \tau^{-4}) \dot\tau = \gamma^{-1} [ (\dot\gamma + \gamma) (A_3 + \tau A_4) - A_4 ]\,.
\end{align*}
Combination with~\eqref{le dot tau} yields
$$
(A_4 - \tau^{-4}) [ (\dot\gamma + \gamma) \tau - A_3 ] = (\dot\gamma + \gamma) (A_3 + \tau A_4) - A_4\,,
$$
hence
\begin{align}\label{le simplification of the coefs 2}
{\dot\gamma}+ \gamma =\tau^{-4} A_3 - A_3 A_4 + A_4 = \tau^{-4} A_3 +\tau^{-3} A_4\,.
\end{align}

Thus, we can write the left side of~\eqref{le lemma coeffs 1} as
\begin{align} \label{claim 1 left}
  (\tau^{-4} A_3 + \tau^{-3} A_4) \tau^{-2} +& \big( \tau^{-3} A_3 + \tau^{-2} A_4 - 1 \big) A_3 =  (\tau^{-3} A_3 + \tau^{-2} A_4) - A_3 \nonumber\\
& =  (\tau^{-3} A_3 + \tau^{-2} A_4) - A_3 (A_3 + \tau^{-3}) \nonumber\\
&= (\tau^{-2} A_4 - A_3^2) \,.
\end{align}
This proves~\eqref{le lemma coeffs 1}. Similarly, we have for the left side of~\eqref{le lemma coeffs 2}
\begin{align} \label{claim 1 right}
&  (\tau^{-4} A_3 + \tau^{-3} A_4) A_3 + \big( \tau^{-3} A_3 + \tau^{-2} A_4 - 1 \big) A_4  \nonumber \\
&\qquad=  (\tau^{-4} A_3 + \tau^{-3} A_4) A_3 + \big( \tau^{-3} A_3 + \tau^{-2} A_4 - (A_3 + \tau^{-3})^2 \big) A_4\nonumber \\
&\qquad=  \tau^{-4} A_3^2 + \tau^{-2} A_4^2 - A_3^2 A_4 - \tau^{-6} A_4 \nonumber\\
&\qquad= (A_4 - \tau^{-4}) (\tau^{-2} A_4 - A_3^2) \,.
\end{align}
This proves~\eqref{le lemma coeffs 2} and hence finishes the proof of Lemma~\ref{le lemma coeffs}.
\end{proof}	
Having proven Lemma~\ref{le lemma coeffs}, we can conclude the proof of Lemma~\ref{getting rid of the rest}.
\end{proof}

 \label{appendixIII}
\end{appendix}

\section*{Acknowledgements}

We thank Horng-Tzer Yau for numerous discussions and remarks. We are grateful to Ben Adlam, Jinho Baik, Zhigang Bao, Paul Bourgade, L\'{a}szl\'o Erd\H{o}s and Antti Knowles for comments.


\begin{thebibliography}{}

\bibitem{AGZ} Anderson, G.\ W., Guionnet, A., Zeitouni, O.: \emph{An Introduction to Random Matrices}, Cambridge University Press (2010).

\bibitem{A1} Anderson, G.\ W.:  \emph{Convergence of the Largest Singular Value of a Polynomial in Independent Wigner Matrices}, Ann. Probab. \textbf{41}, 2103-2181 (2013).

\bibitem{BaiS} Bai, Z.; Silverstein, J.\ W.: \emph{Spectral Analysis of Large Dimensional Random Matrices}, Springer (2009).

\bibitem{BBP} Baik, J., Ben Arous, G., P\'ech\'e, S.: \emph{Phase Transition of the Largest Eigenvalue for Nonnull Complex Sample Covariance Matrices}, Ann. Probab. \textbf{33}, 1643-1697 (2005).

\bibitem{BS} Baik, J., Silverstein, J.\ W.: \emph{Eigenvalues of Large Sample Covariance Matrices of Spiked Population Models}, J. Multiv. Ana. \textbf{97}, 1382-1408 (2006).

\bibitem{BPZ2} Bao, Z., Pan, G., Zhou, W.: \emph{Universality for the Largest Eigenvalue of Sample Covariance Matrices with General Population},  Ann.\ Statist. \textbf{43.1}, 382-421 (2015).

\bibitem{Bia} Bianchi, P., Debbah, M., Ma\"{i}da, M., Najim, J.: \emph{Performance of Statistical Tests for Single-Source Detection Using Random Matrix Theory}, IEEE Trans. Information Theory \textbf{57.4}, 2400-2419  (2011).

\bibitem{BV1} Bloemendal, A., Vir\'ag, B.: \emph{Limits of Spiked Random Matrices I}, Probab. Theory Relat. Fields \textbf{156}, 795-825 (2013).

\bibitem{BV2} Bloemendal, A., Vir\'ag, B.: \emph{Limits of Spiked Random Matrices II}, arXiv:1109.3704 (2011).

\bibitem{BKYY1} Bloemendal, A., Erd\H{o}s, L., Knowles, A., Yau, H.-T., Yin, J.: \emph{Isotropic Local Laws for Sample Covariance and Generalized Wigner Matrices}, Electron. J. Probab, \textbf{19}, 1-53 (2014).

\bibitem{BKYY} Bloemendal, A., Knowles, A., Yau, H.-T., Yin, J.: \emph{On the Principal Components of Sample Covariance Matrices}, arxiv:1404.0788 (2014).

\bibitem{C} Chatterjee, S.: \emph{A Generalization of the Lindeberg Principle}, Ann. Prob. \textbf{34}, 2061–2076 (2006).

\bibitem{Ek} El Karoui, N.: \emph{Tracy-Widom Limit for the Largest Eigenvalue of a Large Class of Complex Sample Covariance Matrices}, Ann. Probab. \textbf{35}, 663-714 (2007).

\bibitem{EKY} Erd\H{o}s, L., Knowles, A., Yau, H.-T.: \emph{Averaging Fluctuations in Resolvents of Random Band Matrices}, Ann. Henri Poincaré \textbf{14}, 1837-1926 (2013). 

\bibitem{EKYY1} Erd\H{o}s, L., Knowles, A., Yau, H.-T., Yin, J.: \emph{Spectral Statistics of Erd\H{o}s-R\'enyi Graphs I: Local Semicircle Law}, Ann. Probab. \textbf{41}, 2279-2375 (2013).

\bibitem{EKYY4} Erd\H{o}s, L., Knowles, A., Yau, H.-T., Yin, J.: \emph{The Local Semicircle Law for a General Class of Random Matrices}, Electr. J. Prob. \textbf{18}, 1-58 (2013).

\bibitem{ESYY} Erd\H{o}s, L., Schlein, B., Yau, H.-T., Yin, J.: \emph{The Local Relaxation Flow Approach to Universality of the Local Statistics for Random Matrices}, Annales Inst. H. Poincar\'e (B) Probability and Statistics \textbf{48}, 1-46 (2012).

\bibitem{EYYBernoulli}  Erd\H{o}s, L., Yau, H.-T., Yin, J.: \emph{Universality for Generalized Wigner Matrices with Bernoulli Distribution}, J. Combinatorics \textbf{1}, 15-85 (2011).

\bibitem{EYY1}  Erd\H{o}s, L., Yau, H.-T., Yin, J.: \emph{Bulk Universality for Generalized Wigner Matrices}, Probab. Theory Relat. Fields \textbf{154}, 341-407 (2012).

\bibitem{EYY} Erd\H{o}s, L., Yau, H.-T., Yin, J.: \emph{Rigidity of Eigenvalues of Generalized Wigner Matrices}, Adv. Math. \textbf{229}, 1435-1515 (2012).

\bibitem{FP} F\'{e}ral, D., P\'{e}ch\'{e}, S.: \emph{The Largest Eigenvalues of Sample Covariance Matrices for a Spiked Population: Diagonal Case}, J. Math. Phys. \textbf{50}, 073302 (2009).

\bibitem{FS} Feldheim, O.\ N., Sodin, S.: \emph{A Universality Result for the Smallest Eigenvalues of Certain Sample Covariance Matrices}, Geom. Funct. Anal. \textbf{20}, 88-123 (2010).

\bibitem{F} Forrester, P.: \emph{The Spectral Edge of Random Matrix Ensembles}, Nucl. Phys. B \textbf{402}, 709-728 (1993).

\bibitem{F1} Forrester, P.: \emph{Probability Densities and Distributions for Spiked and General Variance Wishart $\beta$-Ensembles}, Random Matrices: Theory and Applications \textbf{2.04} (2013).

\bibitem{Gir} Girko, V.\ L.: \emph{The Circular Law}, Teor. Veroyatnost. i Primenen. \textbf{29}, 669–679  (1984).



\bibitem{J} Johansson, K.: \emph{Shape Fluctuations and Random Matrices}, Comm. Math. Phys. \textbf{209}, 437-476 (2000).

\bibitem{Jo} Johnstone, I.\ M.: \emph{On the Distribution of the Largest Eigenvalue in Principal Component Analysis}, Ann. Statist. \textbf{29}, 295-327 (2001).

\bibitem{Jon2}  Johnstone, I. M.: \emph{High Dimensional Statistical Inference and Random Matrices}, 
in \emph{International Congress of Mathematicians I}, 307–333. Eur. Math. Soc.,
Z\"urich (2007).

\bibitem{kay} Kay, S.: \emph{Fundamentals of Statistical Signal Processing, Vol. II: Detection Theory}, Prentice Hall (1998).

\bibitem{KY} Knowles, A., Yin, J.: \emph{Anisotropic Local Laws for Random Matrices}, arXiv:1410.3516 (2014).

\bibitem{LP} Lytova, A., Pastur, L.: \emph{Central Limit Theorem for Linear Eigenvalue Statistics of Random Matrices with Independent Entries}, Ann. Proba. \textbf{37.5}, 1778-1840 (2009).

\bibitem{LSSY} Lee, J.\ O., Schnelli, K., Stetler, B., Yau, H.-T.: \emph{Bulk Universality for Deformed Wigner Matrices}, arXiv:1405.6634 (2014).

\bibitem{LSY} Lee, J. O., Schnelli, K.: \emph{Edge Universality for Deformed Wigner Matrices}, arXiv:1407.8015 (2014).

\bibitem{Mo} Mo, M. Y.: \emph{Rank 1 Real Wishart Spiked Model}, Comm. on Pure and Appl. Math. \textbf{65}, 1528-1638 (2012).

\bibitem{MP} Marchenko, V. A., Pastur, L.: \emph{Distribution of Eigenvalues for Some Sets of Random Matrices}, Math. USSR Sb. \textbf{1}, 457-483 (1967).

\bibitem{Mbook} Mehta, M.\ L.: \emph{Random Matrices}, Third Edition, Academic Press, New York (1991).

\bibitem{EN} Nadakuditi, R.\ R., Edelman, A.: \emph{Sample Eigenvalue Based Detection of High-dimensional Signals in White Noise Using Relatively Few Samples}, IEEE Trans.\ Signal Processing \textbf{56.7}, 2625-2638 (2008).

\bibitem{RRNJWS} Nadakuditi, R.\ R., Silverstein, J.\ W.: \emph{Fundamental Limit of Sample Generalized Eigenvalue Based Detection of Signals in Noise Using Relatively Few Signal-Bearing and Noise-Only Samples}, IEEE J.\ Sel.\ Top.\ Signal Process \textbf{4.3}, 468-480,  (2010).

\bibitem{O} Onatski, A.: \emph{The Tracy-Widom Limit for the Largest Eigenvalues of Singular Complex Wishart Matrices},  Ann. Appl. Probab. \textbf{18}, 470-490 (2008).

\bibitem{O2} Onatski, A.: \emph{Testing Hypotheses about the Number of Factors in Large Factor Models}, Econometrica \textbf{77.5}, 1447-1479  (2009).

\bibitem{Pe1} P\'ech\'e, S.: \emph{Universality Results for the Largest Eigenvalues of Some Sample Covariance Matrix Ensembles}, Probab. Theory Relat. Fields \textbf{143}, 481-516 (2009).

\bibitem{PY} Pillai, N.\ S., Yin, J.: \emph{Universality of Covariance Matrices}, Ann. Appl. Probab. \textbf{24}, 935-1001 (2014).

\bibitem{SC} Silverstein, J.\ W., Choi, S.\ I.: \emph{Analysis of the Limiting Spectral Distribution of Large Dimensional Random Matrices.}, J. Multiv. Anal. \textbf{54}, 175-192 (1995).

\bibitem{So} Soshnikov, A.: \emph{A Note on Universality of the Distribution of the Largest Eigenvalues in Certain Sample Covariance Matrices}, J. Stat. Phys. \textbf{108}, 1033-1056 (2002).

\bibitem{TV2} Tao, T., Vu, V.: \emph{Random Matrices: Universality of Local Eigenvalue Statistics up to the Edge}, Comm. Math. Phys. \textbf{298}, 549-572 (2010).

\bibitem{TW1} Tracy, C., Widom, H.: \emph{Level-Spacing Distributions and the Airy Kernel}, Commun. Math. Phys. \textbf{159}, 151-174 (1994).

\bibitem{TW2} Tracy, C., Widom, H.: \emph{On Orthogonal and Symplectic Matrix Ensembles}, Commun. Math. Phys. \textbf{177}, 727-754 (1996).

\bibitem{Vi} Vinogradova, J., Couillet, R., Hachem, W.: \emph{Statistical Inference in Large Antenna Arrays Under Unknown Noise Pattern}, IEEE Trans.\ Signal Processing \textbf{61.22}, 5633-5645 (2013).

\bibitem{VDN} Voiculescu, D., Dykema, K. J., Nica, A.: \emph{Free Random Variables: A Noncommutative Probability Approach to Free Products with Applications to Random Matrices, Operator Algebras and Harmonic Analysis on Free Groups}, American Mathematical Society (1992).

\bibitem{Wa} Wang, K.: \emph{Random Covariance Matrices: Universality of Local Statistics of Eigenvalues up to the Edge}, Random matrices: Theory and Applications \textbf{1.01} (2012).


\end{thebibliography}
\end{document}